\documentclass[12pt,reqno,oneside]{amsart} 
\usepackage[a4paper,width=15cm,height=24cm]{geometry}
\usepackage{amsmath}
\usepackage{amssymb}
\usepackage{amsthm}
\usepackage{graphicx}
\usepackage{xcolor}
\usepackage{hyperref}

\allowdisplaybreaks
\newtheorem{theorem}{Theorem}[section]
\newtheorem{lemma}[theorem]{Lemma}
\newtheorem{proposition}[theorem]{Proposition}

\theoremstyle{definition}
\newtheorem{definition}[theorem]{Definition}

\theoremstyle{remark}
\newtheorem{remark}[theorem]{Remark}

\numberwithin{equation}{section}

\newcommand{\eps}{\varepsilon}
\newcommand{\Z}{\mathbb{Z}}
\newcommand{\R}{\mathbb{R}}

\newcommand{\Q}{\mathbb{Q}}
\newcommand{\I}{\mathcal{I}}

\newcommand{\IP}{\mathbb{P}}
\newcommand{\IE}{\mathbb{E}}

\newcommand{\heta}{{\hat \eta}}
\newcommand{\hpsi}{{\hat \psi}}

\newcommand{\e}{\mathbf{e}}

\newcommand{\capa}{\mathop{\mathrm{cap}}}
\newcommand{\supp}{\mathop{\mathrm{supp}}}
\newcommand{\dist}{\mathop{\mathrm{dist}}}
\newcommand{\Ran}{\mathop{\mathrm{Ran}}}

\newcommand{\diam}{\mathop{\mathrm{diam}}}
\newcommand{\lf}{\lfloor}
\newcommand{\rf}{\rfloor}

\newcommand{\1}[1]{{\bf 1}{\{#1\}}}

\newcommand{\se}{\mathfrak{s\kern-0.1em.\kern-0.13em{}e\kern-0.1em.}}

\def\bbone{\boldsymbol 1}

\begin{document}

\title{On the internal distance in the interlacement set}
\author{Ji\v{r}\'\i{} \v{C}ern\'y}
\address{Department of Mathematics,
  University of Vienna,
  Nordbergstra\ss e 15,
  1090 Vienna, Austria}
\email{jiri.cerny@univie.ac.at}

\author{Serguei Popov}
\address{Department of Statistics,
  Institute of Mathematics, Statistics and Scientific Computation,
  University of Campinas--UNICAMP,
  rua S\'ergio Buarque de Holanda 651, 13083--859, Campinas SP,
  Brazil}
\email{popov@ime.unicamp.br}


\begin{abstract}
  We prove a shape theorem for the internal (graph) distance on the
  interlacement set $\I^u$ of the random interlacement model on
  $\mathbb Z^d$, $d\ge 3$. We provide large deviation estimates for the
  internal distance of distant points in this set, and use these
  estimates to study the internal distance on the range of a simple
  random walk on a discrete torus.
\end{abstract}

\maketitle
\section{Introduction and the results} 
\label{s_intro}
We study properties of the interlacement set~$\I^u$ of the random
interlacement model. We are mainly interested in its connectivity
properties, in particular in the internal distance (sometimes called the
  chemical distance) on the interlacement cluster.

The random interlacement model was introduced in~\cite{Szn10} in order to
describe the microscopic structure in the bulk which arises when studying
the disconnection time of a discrete cylinder or the vacant set of random
walk on a discrete torus. It can be informally described as a dependent
site percolation on $\mathbb Z^d$, $d\ge 3$, which is `generated' by a
Poisson cloud of independent simple random walks whose intensity is
driven by a non-negative multiplicative parameter $u$. The set covered by
these random walks is called the \emph{interlacement set at level $u$} and is
denoted by $\mathcal I^u$. As the precise definition of $\mathcal I^u$ is
rather lengthy, we postpone it to Section~\ref{s_preliminaries} and state
our results first.

Let $\mathbb P_0^u = \mathbb P[ \,\cdot\, | \, 0\in \mathcal I^u] $ be
the conditional distribution  given that
the origin is in the interlacement set $\I^u$.
For $x,y\in\I^u$ we define $\rho_u(x,y)$ to be the internal distance
between~$x$ and~$y$ within the interlacement set~$\I^u$:
\begin{equation*}
  \begin{split}
    \rho_u(x,y) = \min\{n&: \text{there exist }x_0,x_1,\ldots,x_n\in\I^u
      \text{ such that }x_0=x, x_n=y, \\
      &\text{ and }\|x_k-x_{k-1}\|_1=1
      \text{ for all }k=1,\ldots,n\},
  \end{split}
\end{equation*}
where $\|\cdot\|_1$ denotes the $\ell_1$-norm in~$\Z^d$. As we shall see
below, the set $\I^u$ is $\IP$-a.s.\ connected for all~$u$, so
$\rho_u(x,y)<\infty$ for all $u>0$ and $x,y\in\I^u$. Assuming that
$x\in\I^u$, let $\Lambda^u(x,n) = \{y\in\I^u: \rho^u(x,y)\leq n\}$ be the
ball centred at~$x$ with radius~$n$ in the  internal distance. We
abbreviate $\Lambda^u(n):=\Lambda^u(0,n)$.

The first main result of this paper is the shape theorem for large balls
in the internal distance.
\begin{theorem}
  \label{t_shape}
  For every $u>0$ and $d\ge 3$
  there exists a compact convex set $D_u\subset\R^d$ such that
  for any $\eps>0$ there exists a
  $\IP_0^u$-a.s.\ finite random variable~$N$ such that
  \[
    \big((1-\eps)nD_u\cap \I^u\big) \subset \Lambda^u(n)
    \subset  (1+\eps)nD_u
  \]
  for all $n\geq N$.
\end{theorem}

\begin{remark}
  Clearly, the set $D_u$ is symmetric under rotations and
  reflections of~$\Z^d$ and
  $D_u\subset\{x\in\R^d:\|x\|_1\leq 1\}$ for all~$u$.
  It is straightforward to show that
  $D_u\to \{x\in\R^d:\|x\|_1\leq 1\}$ as $u\to\infty$;
  it would be interesting, however, to be able to say something about
  the behaviour of~$D_u$ when $u\to 0$ (e.g., does the shape become
    close to the Euclidean ball, and what can be said
    about the size of~$D_u$ as $u\to 0$?).
\end{remark}

The key technical step in the proof of Theorem~\ref{t_shape} is a fact
(which is of independent interest) that the distance within the
interlacement cluster should typically be of the same order as the usual
distance.
\begin{theorem}
  \label{t_ap}
  For every $u>0$ and $d\ge 3$ there exist
  constants $C,C' < \infty$ and
  $\delta \in (0,1)$ such that
  \begin{equation*}
    \IP^u_0[\text{there exists }x\in \I^u\cap[-n,n]^d
      \text{ such that } \rho_u(0,x)>C  n]
    \leq C' e^{-n^{\delta}}.
  \end{equation*}
\end{theorem}

A corresponding result for the Bernoulli percolation on $\mathbb Z^d$ was
proved by Antal and Pisztora; in their case the constant $\delta $ equals
one and is optimal, see \cite[Theorem~1.1]{AP96}. We did not try to
optimise the constant~$\delta$ in Theorem~\ref{t_ap}.

\begin{remark}
  It is trivial to replace $\mathbb P^u_0$ by $\mathbb P$ in
  Theorems~\ref{t_shape} and~\ref{t_ap}. To this end it is only necessary
  to extend $\rho_u(x,y)$ to all $x,y\in \mathbb Z^d$ by setting
  \begin{equation*}
    \rho_u(x,y)=\rho_u(x^u,y^u),
  \end{equation*}
  where $x^u$ (respectively, $y^u$) is the closest point to~$x$
  (respectively, $y$) on $\mathcal I^u$ (one can choose the rule how ties
    are broken in any convenient translational-invariant way).
\end{remark}

\medskip

The methods used to show Theorem~\ref{t_shape} also imply the following
result.
\begin{theorem}
  \label{t_connected}
  It holds that $\IP[\I^u \text{ is connected for all }u>0] = 1$.
\end{theorem}
Previously it was known that for every fixed $u>0$, the set~$\I^u$ is
$\mathbb P$-a.s.\ connected (see (2.21) in~\cite{Szn10}); the above
theorem means that $\mathbb P$-a.s.\ there are no `exceptional values' of
the parameter~$u$. Remark also that much more is known about the
connectivity of $\mathcal I^u$ for fixed $u$, see \cite{PT11,RS12}.

Theorems~\ref{t_shape} and~\ref{t_ap} indicate that the interlacement
set~$\mathcal I^u$ looks at large scales very much like~$\mathbb Z^d$. In
the same direction, R\'ath and Sapozhnikov recently proved that the
interlacement set $\mathcal I^u$ percolates in slabs~\cite{RS11b}, and
that random walk on~$\mathcal I^u$ is transient~\cite{RS11}.

Theorem~\ref{t_ap} can be also used to answer a related question: `How
much the range of the random walk on the torus resembles the torus?' To
this end we consider $(X_k)_{k\in \mathbb N}$ to be a simple random walk
on the discrete $d$-dimensional torus of size $N$,
$\mathbb T_N^d=(\mathbb Z/N\mathbb Z)^d$, and write $P^N$ for its law
when started from the uniform distribution. We let $\I_N^u$ to denote the
range of the random walk up to time $uN^d$,
\begin{equation*}
  \I_N^u = \{X_0,\dots, X_{\lf u N^d\rf}\}.
\end{equation*}
Let $\rho^u_N(x,y)$ be the minimal distance of $x,y\in \I_N^u$ within
$\I_N^u$, defined similarly as~$\rho_u$, and let $d_N(x,y)$ be their
usual graph distance on the torus.

\begin{theorem} For large enough $\bar C$ and $\gamma $, we have
  \label{t_torus}
  \begin{equation*}
    \lim_{N\to\infty} P^N\big[\rho^u_N(x,y)\le \bar C d_N(x,y)
      \text{ for all }
      x,y\in \I_N^u \text{ such that } d_N(x,y)\ge \ln^{\gamma} N\big]=1.
  \end{equation*}
\end{theorem}

This theorem improves the result of Shellef~\cite{She10}, where a similar
claim was proved for $\bar C$ growing very slowly with~$N$ using entirely
different methods.  More
precisely, \cite{She10} requires $\bar C=\ln^{(k)}N$ where $\ln^{(k)}$
is the $k$-times iterated logarithm, $k\ge 1$ being arbitrary. On the
other hand, Shellef needs $\gamma=5d$ only; we do not have control on the
size of this constant.

The main difficulty of the paper stems in proving our results for $d<5$,
in particular for $d=3$. In fact, for $d\ge 5$ there is a rather simple
argument, based on the results of~\cite{RS11}, which shows Theorem~\ref{t_ap}
with $\delta =1$, and which we sketch in the Appendix. This argument uses
the fact that for $d\ge 5$ the random interlacement restricted to a
thick-enough two-dimensional slab dominates in some sense the standard
Bernoulli percolation, which allows an application of~\cite{AP96}.
Heuristically, in large dimensions it is possible to construct `long straight
connections' within~$\I^u$ locally, independently of the connections in
other places.

It seems that this argument cannot be extended to $d<5$. It is much harder
to construct the straight connections locally in an independent manner.
This we do in Section~\ref{s:ap}, where we dominate the internal distance
between the origin and the point $(n,0,\dots,0)$ by the sum of a sequence of
random variables with a finite range of dependence and stretched
exponential tails, cf.~\eqref{sum_T} below. To obtain the finite range of
dependence, we should show that connections within
a large box of size~$m$ can be constructed using less than $\Theta(m^{d-2})$ random walk trajectories (which is
  the typical number of random walks intersecting this box;
here and in the sequel we write $f(m)=\Theta(g(m))$
when for positive constants $c_1,c_2$ we have
 $c_1g(m)\leq f(m) \leq c_2 g(m)$ for all~$m$). In fact, in
Proposition~\ref{p_connectedness}
we will show that a `backbone' of~$\I^u$ in this box can be constructed using $\Theta(m^{d-2-h})$ trajectories
only, $h<2/d$. This also means that for every $u>0$ the
interlacement set~$\I^u$ is `largely supercritical', that is it remains
locally connected, even when considerably thinned.

The paper is organised as follows. After introducing the notation in
Section~\ref{s_preliminaries}, we collect in Section~\ref{s_estimates}
some estimates on the hitting probabilities of sets and on the range of
the simple random walk. Section~\ref{s_intersections} contains the key
technical result of this paper, Proposition~\ref{p_connectedness}. This
proposition roughly states that all points in (a possibly thinned version
  of) the set $\I^u$ within box of size $n$ are at internal distance $n^2$,
with a very high probability. Using this proposition, in
Section~\ref{s_connected}, we give a short proof of
Theorem~\ref{t_connected}. Sections~\ref{s:ap}--\ref{s:torus} contain the
proofs of Theorems~\ref{t_ap}, \ref{t_shape}, and~\ref{t_torus}.

\medskip

\textit{Acknowledgements.} The authors would like to thank Augusto
Teixeira for many useful discussions, and Bal\'azs R\'ath for pointing
out Shellef's paper~\cite{She10}. The work of Serguei Popov was partially
supported by CNPq (301644/2011--0) and FAPESP (2009/52379--8).

\section{Preliminaries} 
\label{s_preliminaries}

In this section we fix the notation and recall the definition of the
random interlacement model.

Let $\mathbb N=\{0,1,\dots\}$ be the set of natural numbers. We denote
with $\e_1,\ldots,\e_{d}$ the coordinate vectors in~$\Z^d$, and write
$\|\cdot\|,\|\cdot\|_1, \|\cdot\|_\infty$ for the Euclidean, $\ell_1$,
and $\ell_\infty$ norms correspondingly.  We use $B(x,r)$ to denote the
closed $\|\cdot\|_\infty$-ball centred at $x$ with radius $r$, and
abbreviate $B(r):=B(0,r)$. We say that~$A\subset \mathbb Z^d$ is
connected if for any $x,y\in A$ there is a nearest-neighbor path that
lies fully inside~$A$ and connects~$x$ to~$y$. We write $|A|$ for the
cardinality of~$A$, $\diam(A)=\max_{x,y\in A} \|x-y\|_\infty$ for its
diameter in $\ell_\infty$-norm,  and
$\partial A=\{x\in A: \exists y \in A^c, \|x-y\|=1\}$ for its internal
boundary.

Let us write $P_x$ for the law of a discrete-time simple random walk
$(X_n)_{n\in \mathbb N}$ on $\mathbb Z^d$ started from $x$. For
$A\subset \mathbb Z^d$ we denote with $H_A$, $\tilde H_A$ and $T_A$ the
entrance time in~$A$, the hitting time of~$A$, and the exit time from~$A$:
\begin{equation}
  \label{def_entrance}
  \begin{split}
    H_A&=\inf\{n\ge 0: X_n\in A\}, \\
    \tilde H_A&=\inf\{n\ge 1: X_n\in A\},\\
    T_A&=\inf\{n\ge 0: X_n\notin A\}.
  \end{split}
\end{equation}
Given  $A\subset \mathbb Z^d$ finite, we define the equilibrium measure of~$A$ by
\begin{equation*}
  e_A(x)=P_x[\tilde H_A=\infty]\bbone_A(x)
\end{equation*}
and denote by
$\capa(A)=\sum_{x\in A}e_A(x)$ its total mass.

We now recall the definition of the random interlacement from
\cite{Szn10}. In order to do this we need to introduce another notation
which is, however, mostly used only locally. Let~$W$ be the space of
doubly-infinite nearest-neighbour trajectories in~$\Z^d$  which tend to
infinity at positive and negative infinite times, and let $W^\star$ be
the space of equivalence classes of trajectories in~$W$ modulo
time-shift. (These spaces are equipped with $\sigma $-algebras
$\mathcal W$, $\mathcal W^\star$ as in (1.2), (1.10) of~\cite{Szn10}.)
The random interlacement is defined via a Poisson point process taking
values in the space $\Omega $ of point measures on the space
$W^\star\times [0,\infty)$ with the intensity measure $\nu \otimes du$.
We denote by $\mathbb P$ the law of this process.

To describe the measure~$\nu$ appearing in the intensity of the Poisson
point process, for $A\subset \mathbb Z^d$, $u\ge 0$, we denote by~$\mu_A^u$
the mapping from~$\Omega$ to the space of point measures on~$W$
which selects from $\omega \in \Omega $ the trajectories with labels
smaller than~$u$ intersecting~$A$ and parametrises them so that they
enter~$A$ at time~$0$. Formally, for
$\omega =\sum_{i\ge 0}\delta_{ (w^\star_i,u_i)}\in \Omega $ ,
$w^\star_i\in W^\star$, $u^i\ge 0$, we define
\begin{equation}
  \label{e:muKu}
  \mu_K^u(\omega)=\sum_{i\ge 0}\delta_{s_A(w^\star_i)}
  \bbone\{\Ran(w^\star_i)\cap A\neq\emptyset,u_i\le u\},
\end{equation}
where $\Ran(w^\star)=\bigcup_{n\in \mathbb Z} w(n)$ for an arbitrary~$w$
in the equivalence class of~$w^\star$, and $s_A(w^\star)$ is the unique
$w\in W$ in this equivalence class such that $w_0\in A$, $w_{-n}\notin A$,
$n>0$. As follows from \cite{Szn10}, Theorem~1.1, the measure~$\nu $ is uniquely
determined by the following two properties which we will frequently use:
\begin{itemize}
  \item For every finite set $A\subset \mathbb Z^d$, under~$\mathbb P$,
  the number $\eta _A^u:=\mu_A^u(\omega )(W)$ of trajectories in~$\omega $
  with labels smaller than~$u$ entering~$A$ has the Poisson distribution
  with parameter $u\capa( A)$.

  \item Let
  $\mu_A^u(\omega )=\sum_{i=1}^{\eta_A^u}\delta_{w_i}$, $w_i \in W$. Then,
  under $\mathbb P$, $w_i$ are i.i.d., independent of $\eta_A^u$, with
  the law given by
  \begin{equation*}
    \mathbb P[(w_i(n))_{n\ge 0}\in F]=\sum_{x\in
      A}\frac{e_A(x)}{e_A(A)}P_x[F],
  \end{equation*}
  for any measurable set $F$ in the space of single-infinite
  nearest-neighbour paths. It means that~$w_i$, restricted to
  non-negative times, are i.i.d.\ simple random walk trajectories started
  from the normalised equilibrium measure $e_A(\cdot)/e_A(A)$.
\end{itemize}

The interlacement set at level~$u$ is then defined as the trace of all
trajectories in~$\omega$ with labels smaller than~$u$,
\begin{equation*}
  \I^u(\omega )=\bigcup_{i\ge 0} \Ran(w^\star_i)\bbone \{u_i\le u\}.
\end{equation*}

We now explain the conventions for the use of constants in this paper. We
denote by
$C,C_1, C_1', C_2, \ldots$ the
`global' constants, that is, those that are used all along the paper and
by $c,c',c_1, c_2,c_3,\ldots$ the `local' constants, that is, those that
are used only in the small neighbourhood of the place where they appear
for the first time. For the local constants, we restart the numeration
either in the beginning of each subsection or in the beginning of each
long proof. All these constants are positive and finite and may depend on
dimension, $u$, and other quantities that are supposed to be fixed;
usually we omit expressions like `there exist positive constants $c_1,c_2$
such that \dots' and just directly insert $c$'s to the formulas.

Also, the reader will notice that very frequently in this paper the
probability of events (indexed by some integer parameter, say, $n$) will
happen to be bounded from above by~$e^{-cn^{\delta}}$ or from below
by~$1-e^{-cn^{\delta}}$, where~$\delta$ is typically (but not necessarily)
between~$0$ and~$1$. So, we decided to use the following definition:
\begin{definition}
  \label{def_s.e.}
  We say that $f(n)$ is s.e.-small (where s.e.\ stands for
    `stretched-exponentially') if for all~$n\geq 1$ it holds that
  \[
    0 \leq f(n) \leq c_1 e^{-c_2n^{c_3}},
  \]
  and write $f(n)=\se(n)$.
\end{definition}
Observe that $n^c\se(n)=\se(n)$ for any fixed $c>0$. So, it is quite
convenient to use this notation e.g.\ in the following situation: assume
that we have at most~$n^c$ events, each of probability bounded from above
by~$\se(n)$. Then, the probability of their union is~$\se(n)$ as well.

\section{Estimates on hitting probabilities} 
\label{s_estimates}
In this section we collect several estimates on hitting probabilities of
subsets of~$\Z^d$ by random walk trajectories. We recall that $P_x$
denotes the law of the simple random walk $(X_n)_{n\in \mathbb N}$ in
$\Z^d$, $d\geq 3$, starting at $x$. We denote by~$g$ the `stopped' Green
function:
\[
  g(x,y;n) = \sum_{k=0}^n P_x[X_k=y],
\]
and write $g(x,y)$ for $g(x,y;\infty)$.
For the case $d\geq 3$ it holds that $g(x,y)$ is finite for
all $x,y\in\Z^d$, $g(x,y;n)=g(y,x;n)=g(0,y-x;n)$, and, for
all $n\geq \|x-y\|^2$
\begin{align}
  g(x,y;n) &\geq \frac{C_1}{1+\|x-y\|^{d-2}},
  \label{Green_asymp>}\\
  g(x,y)&\leq \frac{C'_1}{1+\|x-y\|^{d-2}},
  \label{Green_asymp<}
\end{align}
for all $x,y\in\Z^d$. The upper bound \eqref{Green_asymp<} follows
directly from Theorem~1.5.4 of~\cite{Law91}. The lower bound
\eqref{Green_asymp>} can be proved easily adapting the proof of the same
theorem.

For $n\geq 0, x\in\Z^d, A\subset\Z^d$, let
\[
  q_x(A;n) = P_x[H_A\le n]
\]
be the probability that, starting from~$x$, the simple random walk
enters~$A$ before time~$n$. We use the abbreviation
$q_x(y;n):=q_x(\{y\};n)$ for the hitting probabilities of one-point sets,
and $q_x(A):=q_x(A;\infty)$ for the probability that the simple random
walk ever enters the set~$A$. It is elementary to obtain that for all
$x,y\in\Z^d$ and
$n\geq \|x-y\|^2$ (see e.g.\ Theorem~2.2 of~\cite{AMP02})
\begin{equation}
  \label{prob_to_hit>}\\
  \begin{split}
    q_x(y;n) &\geq \frac{C_2}{1+\|x-y\|^{d-2}},\\
    q_x(y)&\leq \frac{C'_2}{1+\|x-y\|^{d-2}}.
  \end{split}
\end{equation}

Next, for $x \in \Z^d$ and a finite set $A \subset \Z^d$, define
\[
  g(x,A;n) = \sum_{y\in A} g(x,y;n).
\]
Clearly, $g(x,A;n)$ is the expected number of visits to~$A$ up to
time~$n$, starting from~$x$. As before, we set
$g(x,A):=g(x,A;\infty)$.

The following lemma will be used repeatedly to estimate the hitting
probabilities:
\begin{lemma}
  \label{l_hit_Green}
  For all $x\in \mathbb Z^d$, finite $A\subset \mathbb Z^d$, and
  $0\leq n\leq \infty$
  \begin{equation}
    \label{hit_Green}
    \frac{g(x,A;n)}{\max_{y\in A}g(y,A;n)} \leq q_x(A;n) \leq
    \frac{g(x,A)}{\min_{y\in A}g(y,A)}.
  \end{equation}
\end{lemma}

\begin{proof}
  Using the definition of $g$ and
  the strong Markov property,
  \begin{align*}
    g(x,A) &=
    \sum_{y\in A} P_x[H_A<\infty, X_{H_A}=y]g(y,A)\\
    & \geq \min_{y\in A}g(y,A)\sum_{y\in A}
    P_x[H_A<\infty, X_{H_A}=y] .
  \end{align*}
  Since $q_x(A;n)\leq q_x(A)=\sum_{y\in A} P_x[H_A<\infty, X_{H_A}=y]$,
  the second inequality in~\eqref{hit_Green} follows.
  The first inequality is then implied by
  \begin{align*}
    g(x,A;n) &= \sum_{k=0}^n
    \sum_{y\in A} P_x[H_A=k, X_{H_A}=y]g(y,A;n-k)\\
    &\leq \sum_{k=0}^n\sum_{y\in A}
    P_x[H_A=k, X_{H_A}=y]g(y,A;n)\\
    & \leq \max_{y\in A}g(y,A;n)\sum_{k=1}^n\sum_{y\in A}
    P_x[H_A=k, X_{H_A}=y],
  \end{align*}
  together with
  $q_x(A;n) = \sum_{k=0}^n\sum_{y\in A} P_x[H_A=k, X_{H_A}=y]$.
\end{proof}

Let us use the notation $\ell(x,A)=\max_{y\in A}\|x-y\|_\infty$
for the maximal distance between~$x$ and the points
of~$A$.
Two following simple lemmas contain lower bounds on
hitting probabilities of sets.
\begin{lemma}
  \label{l_diam}
  Suppose that $A$ is a connected finite subset of~$\Z^d$, containing
  at least two sites.
  Then, for all $x\in \Z^d$ and $n\geq (\ell(x,A))^2$,
  \begin{equation*}
    q_x(A;n) \geq
    \begin{cases}
      \displaystyle
      \frac{C_3\diam(A)}
      {(\ell(x,A))^{d-2}\ln \diam(A)},\qquad
      & d=3,
      \medskip
      \\
      \displaystyle
      \frac{C_3\diam(A)}{(\ell(x,A))^{d-2}},
      & d\geq 4.
    \end{cases}
  \end{equation*}
\end{lemma}

\begin{proof}
  Since~$A$ is connected,  it is possible to find (not necessarily
    connected) set $A'\subset A$ with the following properties:
  \begin{itemize}
    \item $|A'|=\diam(A)$,
    \item one can represent
    $A'=\{x_1,\ldots,x_{\diam(A)}\}$ in such a way that
    $\|x_i-x_j\|_\infty\geq |i-j|$ for all $i,j=1,\ldots,\diam(A)$.
  \end{itemize}
  Indeed, it holds that the size of the projection of~$A$ on one of the
  coordinate axes is at least $\diam(A)$ and this projection is an
  interval; then, for all points in the projection pick exactly one
  element of~$A$ that projects there, and erase `unnecessary' points of~
  $A$. Then, by~\eqref{Green_asymp>} we have for any $n\geq (\ell(x,A))^2$
  \[
    g(x,A';n) \geq \frac{C_1\diam(A)}{1+(\ell(x,A))^{d-2}},
  \]
  and, by~\eqref{Green_asymp<},
  for any $y\in A'$,
  \[
    g(y,A';n) \leq \sum_{j=0}^{\diam(A)}\frac{2C'_1}{1+j^{d-2}} \leq
    \begin{cases}
      c_1\ln \diam(A), & d=3,\\
      c_1, & d\geq 4.
    \end{cases}
  \]
  Since $q_x(A;n)\geq q_x(A';n)$ for all~$n$, the claim follows
  from Lemma~\ref{l_hit_Green}.
\end{proof}

The previous lemma works well for sparse connected sets. For more densely
packed sets we need another estimate:

\begin{lemma}
  \label{l_ball}
  For all $x\in\Z^d$, finite $A\subset\Z^d$ containing
  at least two sites, and all $n\geq (\ell(x,A))^2$,
  \begin{equation*}
    q_x(A;n) \geq \frac{C_4|A|^{1-\frac{2}{d}}}{(\ell(x,A))^{d-2}}.
  \end{equation*}
\end{lemma}

\begin{proof}
  Again using~\eqref{Green_asymp>}, we have for any $n\geq (\ell(x,A))^2$
  \[
    g(x,A;n) \geq \frac{C_1|A|}{1+(\ell(x,A))^{d-2}}.
  \]
  To obtain an upper bound on $g(y,A;n)$ for $y\in A$, we
  observe that
  \[
    |\{x\in\Z^d: \|x\|\in[k,k+1)\}| = \Theta(k^{d-1}).
  \]
  So, using~\eqref{Green_asymp<}, we have for $y\in A$
  \begin{align*}
    g(y,A;n) &\leq \sum_{z\in A} \frac{C'_1}{1+\|y-z\|^{d-2}}\\
    &\leq \sum_{k=0}^{c_3|A|^{1/d}}
    \frac{C'_1 c_4k^{d-1}}{1+k^{d-2}}\\
    &\leq c_5 |A|^{2/d},
  \end{align*}
  where we have used an obvious worst-case estimate (all the points of~$A$
    are grouped around~$y$, forming roughly a ball of radius
    $\Theta(|A|^{1/d})$) on the passage from the first to the second line of
  the above display. Then, applying Lemma~\ref{l_hit_Green} we conclude
  the proof of Lemma~\ref{l_ball}.
\end{proof}

We end this section by stating a few well-known facts about the behavior
of the set of sites visited by a simple random walk by time~$n$. As we
could not locate suitable references, we also sketch their proofs.
\begin{lemma}
  \label{l_range_SRW}
  Suppose that $d\geq 3$ and let
  $R(n)=\{X_0,\ldots,X_n\}$ be the set of sites visited by a simple random
  walk by time~$n$. Then,
  for any fixed $\alpha_1\in (0,1)$,
  \begin{equation*}
    P\big[n^{1-\alpha_1}\leq
      \diam(R(n^2))\leq n^{1+\alpha_1},
      |R(n^2)|\geq n^{2-2\alpha_1}\big]
    \geq 1-\se(n).
  \end{equation*}
\end{lemma}

\begin{proof}
  The upper bound on the diameter follows from
  any convenient large deviation bound on the
  displacement of the simple random walk (e.g.\ Lemma~1.5.1
of~\cite{Law91}).

  To control the diameter and the number of visited sites from below, we
  use the following simple argument: We divide  the temporal interval
  $[0,n^2]$ into $c^{-1}n^{2\alpha_1}$ subintervals of length
  $cn^{2-2\alpha_1}$, for a large enough~$c$. Clearly, on each of the
  subintervals of length $cn^{2-2\alpha_1}$ the maximal displacement of
  the simple random walk is at least $n^{1-\alpha_1}$ with a constant
  probability, e.g.,\ by the central limit theorem. Noting that  by
  time~$k$ the number of visited sites is at most~$k$, and that the
  expectation of this number is at least $c'k$ 
(it is straightforward to obtain this from~\eqref{Green_asymp>}),
  we deduce that also with at least constant probability%
  \footnote{For any random variable~$\xi$ with $0\leq \xi\leq a$ a.s.\
    and $E \xi\geq b$, it is true that $P[\xi\geq b/2]\geq b/(2a)$.} the
  number of different sites visited by the random walk during a fixed
  temporal interval of length $cn^{2-2\alpha_1}$ is at least
  $n^{2-2\alpha_1}$ (if $c$ is large enough).
  Finally, to estimate the
  probability that the event of interest occurs on at least one of the
  $c^{-1}n^{2\alpha_1}$ subintervals, use the independence. The claim
  then follows easily.
\end{proof}

We also need an estimate on the number of different sites
visited by \emph{several} random walks:
\begin{lemma}
  \label{l_range_many}
  Consider~$k$ independent simple random walks
  $(X^{(1)}_j)_{j\ge 0},\ldots,(X^{(k)}_j)_{j\ge 0}$ started from arbitrary points
  $x^{(1)},\dots, x^{(k)}$, and denote
  $R_j(m)=\{X^{(j)}_0,\ldots,X^{(j)}_m\}$, $j=1,\ldots,k$.
  Assume that $n^{h}\leq k\leq n^{d-2}$ for some fixed $h\in(0,d-2)$.
  Then, for any~$\alpha_3\in(0,h)$ we have
  \begin{equation*}
    P\Big[\Big|\bigcup_{j=1}^k R_j(n^2)\Big|\geq kn^{2-\alpha_3}\Big]
    \geq 1-\se(n).
  \end{equation*}
\end{lemma}

\begin{proof}
  We use a similar argument as in the previous proof.
  We divide the~$k$ walks into
  $c^{-1}n^{\alpha_3}$ groups, each containing $ckn^{-\alpha_3}$
  walks. Consider now the $ckn^{-\alpha_3}$ walks of the, say,
  first group, suppose that they are labelled from~$1$
  to~$ckn^{-\alpha_3}$. Let
  \[
    V=\bigcup_{j=1}^{ckn^{-\alpha_3}} R_j(n^2)
  \]
  be the set of sites visited by the walks from the first group. For
  $y\in\Z^d$, define
  \[
    \zeta(y) = \sum_{j=1}^{ckn^{-\alpha_3}} \1{\|x^{(j)}-y\|\leq n}
  \]
  to be the number of walks of the first group that start at
  distance at most~$n$ from~$y$. By~\eqref{prob_to_hit>},
  using $\zeta (y)\le k\le n^{d-2}$, we have
  \[
    P[y\in V] = 1-\prod_{j=1}^{ckn^{-\alpha_3}}
    \big(1-q_{x^{(j)}}(y;n^2)\big)
    \geq \frac{c'\zeta(y)}{n^{d-2}}.
  \]
  So, if $c$ is large enough
  \begin{align*}
    E|V| &= \sum_{y\in\Z^d}P[y\in V] \\
    & \geq \frac{c'}{n^{d-2}} \sum_{y\in\Z^d}\zeta(y)\\
    & \ge \frac{c'ckn^{-\alpha_3}}{n^{d-2}}
    \big|\{y:\|y\|\leq n\}\big|\\
    &\geq 2kn^{2-\alpha_3}.
  \end{align*}
  Since, trivially, $|V|\leq ckn^{2-\alpha_3}$, it holds that
  $|V|\geq kn^{2-\alpha_3}$ with at least a constant probability.
  As the same
  reasoning applies to each of the $c^{-1}n^{\alpha_3}$ groups,
  the claim of the lemma follows by independence.
\end{proof}

\section{Intersections of random walks} 
\label{s_intersections}

In this section we show that the set of points visited by sufficiently many
walks started in~$B(n)$ is typically well connected; the precise statement of
this fact is contained in Proposition~\ref{p_connectedness}.

To state this proposition we need some notation.
We consider
two
sequences of positive random variables
$\heta^{(n)}_1, \heta^{(n)}_2$ satisfying
$\heta^{(n)}_1\le \heta^{(n)}_2$ and
\begin{align}
  \IP[\heta^{(n)}_1\geq C_5 n^{d-2-h}] &\geq 1-\se(n),
  \label{heta_1}
  \\
  \IP[\heta^{(n)}_2\leq C_6 n^{M}] &\geq 1-\se(n),
  \label{heta_2}
\end{align}
for some $h<\frac{2}{d}$ and $M>0$.
Let $X^{(1)},\ldots,X^{(\heta^{(n)}_2)}$ be
$\heta^{(n)}_2$ independent simple random walks starting from
some sites $x^{(1)},\ldots,x^{(\heta^{(n)}_2)}\in B(n)$. We write $P$
for the joint distribution of these walks.
Let $R_k(m)=\{X^{(k)}_0,\ldots,X^{(k)}_m\}$
be the set of different sites visited by
$k$th random walk until time~$m$.  We write $H^k_A$, $\tilde H^k_A$ for
the entrance and hitting time of~$A$ by random walk $X^{(k)}$
(recall~\eqref{def_entrance}).

\begin{definition}
  \label{def_(...)-connected}
  For integers $s,m\geq 1$
  we say that $X^{(i)}$ is $(s,m)$-connected to $X^{(j)}$ if there exist
  a sequence of integers $i=k_0,k_1,\ldots,k_s=j$ such that
  \begin{equation*}
    \begin{split}
      &k_t\leq \heta^{(n)}_1,
      \qquad \text{for all }t=0,\dots,s\\
      &R_{k_t}(m)\cap R_{k_{t-1}}(m)\neq \emptyset,
      \qquad \text{for all }t=1,\dots,s.
    \end{split}
  \end{equation*}
  (We do not indicate the dependence on~$n$ in order to keep the notations
    not too heavy.)
\end{definition}
In words, the definition says that
the trajectories are $(s,m)$-connected if one can go from the
starting point of the $i$th trajectory to the starting point of the $j$th
trajectory within the cluster of the first $\heta^{(n)}_1$ trajectories,
by changing no more than~$s$ times the trajectory,
and using at most~$m$ sites in the beginning of each trajectory.

Let us define for $k\leq \heta^{(n)}_2$ the following set of integers:
\[
  L_k = \big\{m\geq 1:
    \{X^{(k)}_{3mn^2},\ldots,X^{(k)}_{3(m+1)n^2-1}\}\cap B(n)\neq \emptyset \big\},
\]
and let
\begin{equation}
  \label{e:Jn}
  J^{(n)} = \{k\leq \heta^{(n)}_2 : |L_k|=0\}
\end{equation}
be the index set of the walks that do not come back to~$B(n)$
after the time~$3n^2$.

For $d\geq 3$ and $h<\frac{2}{d}$, define
\begin{equation}
  \label{def_beta_d}
  \beta(d,h) := \min\Big\{k\geq 1:
    \frac{dh}{2}+\Big(d-3+h-\frac{dh}{2}\Big)
    \Big(1-\frac{2}{d}\Big)^{k-1}<1\Big\}
\end{equation}
(in fact, this quantity represents the necessary number of steps in the recursive construction used in the proof of Proposition~\ref{p_connectedness}, see~\eqref{def_a_n} 
and~\eqref{recursion_a_n}; at this point we only observe that
$\beta(d,h)$ is finite since $\frac{dh}{2}<1$).

The following proposition plays the key role in this paper:
\begin{proposition}
  \label{p_connectedness}
  Let
  $\heta^{(n)}_1$, $\heta^{(n)}_2$ and $X^{(k)}$,
  $k\le \heta^{(n)}_2$, be as above. Then
  \begin{equation}
    \label{n2connected}
    P\big[\forall i,j\leq \heta^{(n)}_1, X^{(i)}\text{ and } X^{(j)}
      \text{ are $\big(2\beta(d,h)+1,2n^2\big)$-connected}\big]
    \geq 1-\se(n).
  \end{equation}
  Moreover,
  \begin{equation}
    \label{connected_lower_level}
    P\big[\forall i\leq \heta^{(n)}_2 \ \exists j\leq \heta^{(n)}_1
      \text{ such that } R_i(n^2)\cap R_j(2n^2)\neq\emptyset\big]
    \geq 1-\se(n),
  \end{equation}
  and
  \begin{align}
    \label{connected_later_times}
    \begin{split}
      P\big[&\forall i\leq \heta^{(n)}_2\ \forall m\in L_i\
        \exists j\leq \heta^{(n)}_1
        \text{ such that }  \\
        & ~~~\{X^{(i)}_{3mn^2},\ldots,X^{(i)}_{3(m+1)n^2-1}\}
        \cap R_j(2n^2)\neq\emptyset\big]
      \geq 1-\se(n).
    \end{split}
  \end{align}
\end{proposition}

\begin{remark}
  (a) The estimates in the above proposition only depend on
  the number of walks that we consider, they
  are uniform with respect to the choice of the starting positions.

  (b) Typically, when applying Proposition~\ref{p_connectedness} to the
  interlacement set (say, in the ball $B(n)$), the variables
  $\heta^{(n)}_{1}$, $\heta^{(n)}_{2}$ will be of order $n^{d-2}$, so
  that $h=0$. The proposition implies that the model of random
  interlacements is `far from the criticality' with respect to the
  connectedness of the interlacement cluster; we typically need much less
  than $\Theta(n^{d-2})$ walks to ensure that the interlacement set is `well
  connected'.

  (c) In the most important case $h=0$, it holds that
  $\beta(3,0)=1, \beta(4,0)=2, \beta(5,0)=3, \beta(6,0)=4$, but then
  $\beta(7,0)=6$. Comparing this with the results of \cite{RS12,PT11} (where it
    is proved that every two points in $\mathcal I^u$ can be joined by a
    path switching the trajectory at most $(\lceil d/2\rceil -1)$-times)
  indicates that the constants $\beta(d,h)$ are not optimal. The authors
  did not check if the formula~\eqref{def_beta_d} can be further
  simplified, but it is clear that $\beta(d,h)=\Theta(d\ln d)$ as $d\to\infty$.
  In any case, for our needs it is enough to know that $\beta(d,h)$ is
  finite for any $d\geq 3$ and $h<2/d$, and this fact is quite obvious.
\end{remark}

First, let us describe informally the idea of the proof
for the particular case $h=0$
(one may note that there are many similarities with the
  proof of Theorem~3.2 of~\cite{AMP02}, and with techniques used in
  \cite{RS12}).
Consider the random walk~$X^{(1)}$ and run it up to time~$n^2$.
Then, $\diam(R_1(n^2))$ is typically of order~$n$, so any other
random walk $X^{(k)}$ hits the set~$R_1(n^2)$ with probability
at least
of order roughly~$n^{-(d-3)}$ (with logarithmic correction for $d=3$)
by Lemma~\ref{l_diam}. Since there are $\Theta(n^{d-2})$ other
available walks, with high probability $R_1(n^2)$ will be hit
by $\Theta(n)$ \emph{different} other walks. In dimension~$d=3$, running
these $\Theta(n)$ walks
for~$n^2$ time units more after the respective hitting
moments of $R_1(n^2)$ is already enough to meet all the other trajectories
(again applying Lemma~\ref{l_diam}, one obtains that
  the probability that any other trajectory hits none of those walks
  is almost exponentially small in~$n$).
In dimension $d\ge 4$ this argument, however, just barely does not work.

So, what to do in dimension~$4$? Consider those $\Theta(n)$ trajectories
(of length~$n^2$)
that intersect the initial one. Together with the initial
trajectory, they form a connected set of cardinality roughly~$n^3$.
We then apply Lemma~\ref{l_ball} to obtain that a random walk starting
somewhere at the boundary of~$B(n)$ will hit such a set with
probability
at least of order $n^{-2}\times n^{3(1-\frac{2}{d})}$. Since
(recall that now $d=4$) we have $\Theta(n^2)$ walks in total,
typically $\Theta(n^{3(1-\frac{2}{d})})$ of them will hit that set.
Since in four dimensions Lemma~\ref{l_diam} gives lower bound of
order~$n^{-1}$ for the hitting probability of the initial piece
of length~$n^2$ of a generic trajectory and
$3(1-\frac{2}{4})=\frac{3}{2}>1$, running these $\Theta(n^{3/2})$
walks a bit more we meet all the other trajectories with high
probability (see on Figure~\ref{f_inters_RWs} an illustration of the
  proof for $d=4$).
\begin{figure}
  \centering \includegraphics{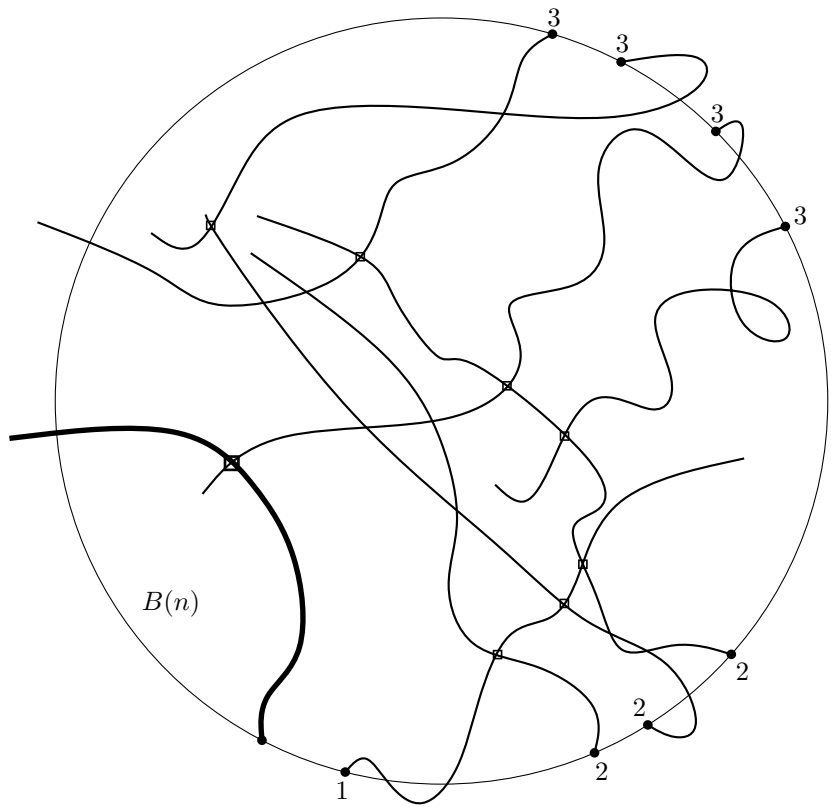} 
  \caption{On the proof
    of Proposition~\ref{p_connectedness} for $d=4$. One considers first
    the trajectory of some particle (labelled here by~`$1$') up to
    time~$n^2$. The trajectories of particles labelled by~`$2$' meet
    the first trajectory (the small boxes indicate the corresponding
      places of first hitting). The particles labelled by~`$3$' then
    hit the set formed by the trajectories with labels~`$1$' and~`$2$'.
    Continuing the trajectories of `$3$'-particles a bit more, one
    finds a `generic' trajectory (the thick one on this picture) with
    very high probability.}
  \label{f_inters_RWs}
\end{figure}

Again, in dimension~$d=5$ this fails since Lemma~\ref{l_diam}
now gives a lower bound of order $n^{-2}$. However, iterating the above
construction, we then obtain roughly
$\Theta(n^{(2+3(1-\frac{2}{d}))(1-\frac{2}{d})})$ independent walks,
and, since
$(2+3(1-\frac{2}{5}))(1-\frac{2}{5})=\frac{57}{25}>2$, these are enough to detect
all the other walks. For any fixed~$d$ one can perform enough iterations
to make this construction work.

If we recursively define the sequence
\begin{equation}
  \label{def_a_n}
  a_1^{(d)}=1, \qquad
  a_{n+1}^{(d)}=(a_n^{(d)}+2)\Big(1-\frac{2}{d}\Big),
\end{equation}
then the necessary number of iterations~$\beta(d,0)$ can be calculated
as follows:
\[
  \beta(d,0)=\min\{k: a_k^{(d)}>d-3\}.
\]
Since it is straightforward to obtain from the
recursion~\eqref{def_a_n} that
\[
  a_n^{(d)}=d-2-(d-3)\Big(1-\frac{2}{d}\Big)^{n-1},
\]
we see
that the above definition of $\beta(d,0)$ agrees
to~\eqref{def_beta_d}.

In order to make the above argument rigorous,
we have to address several issues, for example:
\begin{itemize}
  \item Deal with the dependence of the walks that participate
  in different stages of the above construction. This can be done
  by dividing the walks we use into~$\beta(d,h)$ groups and use one
  group on each stage.

  \item In fact, the trajectories can go back to the ball~$B(n)$
  at later epochs (i.e., much later then $n^2$).
  To prove~\eqref{connected_later_times},
  we have to assure that the random walks constructed on the
  $\beta(d,h)$th stage would meet these pieces of the trajectories
  too,
  otherwise we would have no good control on the distance within
  the interlacement cluster. So, we have to control the `total number
  of returns' (see~\eqref{control_Lm} below).
  In addition, in the above construction we shall use only the walks
  conditioned on not returning
  to~$B(n)$ after time~$3n^2$
  (in order not to be obliged to condition on a too much
    detailed future behaviour of the trajectory).

  \item Finally, all the events described in the informal construction
  should not only be `typical' in some sense, but hold with probability
  at least $1-\se(n)$.
  For that, we need to `adjust' (by sufficiently
    small amounts) the values in the power of~$n$ on each stage.
\end{itemize}

\begin{proof}[Proof of Proposition~\ref{p_connectedness}]
  We start with the formal proof of Proposition~\ref{p_connectedness}. To
  simplify the notation we write $\beta = \beta (d,h)$.
  Recall~\eqref{e:Jn} and define for $m=1,\ldots,\beta$
  \[
    J^{(n)}_{m} = \Big\{k\in J^{(n)}:
      \frac{(m-1)\heta^{(n)}_1}{\beta}
      \leq k <
      \frac{m \heta^{(n)}_1}{\beta}
      \Big\}.
  \]
  Since, clearly, there is a constant $c_4>0$ such that for
  all~$x\in B(n)$ we have
  \begin{equation}
    \label{quit_ball}
    P_x[X_m\notin B(n)\text{ for all }m\geq n^2] > c_4,
  \end{equation}
  we obtain that
  \begin{equation}
    \label{enough_part_psi_d}
    P\big[|J^{(n)}_m|
      \geq c_5n^{d-2-h}
      \text{ for all }m=1,\ldots,\beta\big]
    \geq 1-\se(n).
  \end{equation}
  Inequality~\eqref{quit_ball} further implies that
  that for every~$k, m\ge 1$
  \begin{equation}
    \label{control_Lm}
    \IP[|L_k|>m]\leq e^{-c_7m}.
  \end{equation}

  In the sequel, we will repeatedly use the following observation. For a
  simple random walk~$X$, let $X_{[0,2n^2]}$ be the piece of trajectory
  of the walk~$X$ up to time~$2n^2$. Then there is a constant~$c_8>0$
  such that for any event~$A$ which depends only on the initial piece of
  the trajectory of length~$2n^2$
  \begin{equation}
    \label{cond_behaviour}
    P[X_{[0,2n^2]}\in A \mid X_m\notin B(n) \text{ for all }m>3n^2]
    \geq c_8 P[X_{[0,2n^2]}\in A].
  \end{equation}
  Indeed, to prove~\eqref{cond_behaviour}, we write
  \begin{align*}
    \lefteqn{P[X_{[0,2n^2]}\in A \mid X_m\notin B(n) \text{ for all }m>3n^2] }\\
    &\geq P[X_{[0,2n^2]}\in A , X_m\notin B(n) \text{ for all }m>3n^2] \\
    &= P[X_{[0,2n^2]}\in A] P[X_m\notin B(n) \text{ for all }m>3n^2
      \mid X_{[0,2n^2]}\in A],
    \\&\ge P[X_{[0,2n^2]}\in A] \inf_{x\in \mathbb Z^d}
    P_x\Big[X_m\notin B(n) \text{ for all }m>n^2\Big]
  \end{align*}
  and use~\eqref{quit_ball} 
  to argue that the last
  term is at least of constant order.

  As a last preparatory observation, note that,
  for any~$\eps>0$,
  by Lemma~\ref{l_range_SRW} and the observation following
  Definition~\ref{def_s.e.},
  \begin{equation}
    \label{diam>}
    P\bigg[\,
      \parbox{10cm}{for all $k\leq\heta^{(n)}_2, \diam(R_k(n^2))\geq
        n^{1-\eps}$, \\
        $\diam(\{X^{(k)}_{3mn^2},\ldots,X^{(k)}_{3(m+1)n^2-1}\})\geq n^{1-\eps}$
        for all $m\in L_k$}
      \bigg]
    \geq 1-\se(n).
  \end{equation}

  Let $i_1=\min\{j:j\in J^{(n)}_1\}$, and define
  \[
    V_1 = R_{i_1}\big(n^2\wedge T^{i_1}_{B(2n)}\big).
  \]
   For any $j$ we obtain using Lemma~\ref{l_diam},
  and~\eqref{cond_behaviour} with $A=\{R_j(n^2)\cap V_1\neq\emptyset\}$,
  \begin{equation}
    \label{hit_V1}
    P[R_j(n^2)\cap V_1\neq \emptyset \mid j\in J^{(n)}_1] \geq
    \begin{cases}
      \displaystyle
      \frac{c_9\diam(V_1)}{n\ln n},\qquad & d=3,\\[0.5cm]
      \displaystyle\frac{c_9\diam(V_1)}{n^{d-2}}, & d\geq 4.
    \end{cases}
  \end{equation}
  We introduce the set of indices
  $K_1=\big\{j\in J^{(n)}_1\setminus\{i_1\}:
   R_j(n^2)\cap V_1\neq \emptyset\big\}$.
  By~\eqref{enough_part_psi_d}, 
  \eqref{diam>}, and~\eqref{hit_V1}, using the independence the random
  walks $X^{(j)}$,
  it holds that
  \begin{equation*}
    P[|K_1|\geq n^{1-h-\eps_1}]
    \geq 1-\se(n),
  \end{equation*}
  where $\eps_1:=2\eps$ ($\eps$ is supposed to be sufficiently
    small so that $1-h-\eps_1>0$).

  For $d=3$, everything is ready to finish the proof
  of Proposition~\ref{p_connectedness}, but for other values of~$d$ we
  first need to describe a general step of the construction
  (recall that~$\beta$ steps are necessary).
  Define recursively (recall~\eqref{def_a_n})
  \begin{equation}
  \label{recursion_a_n}
    \begin{split}
      a_1 &= a_1(d,h)=1-h, \\
      a_{n+1} &= a_{n+1}(d,h)=(a_n(d,h)+2)\Big(1-\frac{2}{d}\Big)-h.
    \end{split}
  \end{equation}
  From the above recursion it is straightforward to obtain
  that
  \[
    a_n = d-2-\frac{dh}{2}-\Big(d-3+h-\frac{dh}{2}\Big)
    \Big(1-\frac{2}{d}\Big)^{n-1}.
  \]
  So, with $\beta$ defined by~\eqref{def_beta_d},
  it holds that $a_{\beta}>d-3$.

  Assume that for some $1\leq m\leq \beta-1$ we have constructed
  the connected sets~$V_m\subset\Z^d$ and also
  the sets $K_m\subset J^{(n)}_m$ of indices of the walks which
  hit~$V_m$ before time~$n^2$,
  such that with probability at least $1-\se(n)$
  \begin{equation}
    \label{size_Km}
    |K_m|\geq n^{a_m-\eps_m}.
  \end{equation}
  Then, define
  \[
    V_{m+1}
    = V_m\cup\Big(\bigcup_{j\in
        K_m}R_j\big(2n^2\wedge T^j_{B(2n)}\big) \Big).
  \]
  By Lemma~\ref{l_range_many} (observe that, by~\eqref{cond_behaviour},
its proof still goes through in this situation) 
and~\eqref{size_Km}
  it holds that
  \begin{equation}
    \label{size_Vm}
    P\big[|V_{m+1}|\geq n^{2+a_m-2\eps_m}\big]
    \geq 1-\se(n).
  \end{equation}
  Define
  $K_{m+1}=\{j\in J^{(n)}_{m+1}: R_j(n^2)\cap V_m\neq \emptyset\}$.
  Observe that, by Lemma~\ref{l_ball} and~\eqref{cond_behaviour} with
  $A=\{R(n^2)\cap V_{m+1}\neq \emptyset\}$,
  for any $j\in J^{(n)}_{m+1}$
  \begin{equation}
    \label{hit_Vm+1}
    P[R_j(n^2)\cap V_{m+1}\neq \emptyset]
    \geq \frac{c|V_{m+1}|^{1-\frac{2}{d}}}{n^{d-2}}.
  \end{equation}
  So, using~\eqref{enough_part_psi_d}, 
  \eqref{diam>}, \eqref{size_Vm}, and~\eqref{hit_Vm+1},
  we obtain
  \begin{equation*}
    P\big[|K_{m+1}|\geq
      n^{(2+a_m-2\eps_m)(1-\frac{2}{d})-h-\eps_m}\big]
    \geq 1- \se(n),
  \end{equation*}
  and (for the next induction step) denote
  $\eps_{m+1}=a_{m+1}-(2+a_m-2\eps_m)(1-\frac{2}{d})+\eps_m$,
  so that~\eqref{size_Km} would hold with $m+1$ instead of~$m$.

  Now we describe the last step needed for the proof
  of~\eqref{n2connected}, \eqref{connected_lower_level},
  and~\eqref{connected_later_times}. Assume that on the initial
  step the parameter~$\eps$ was chosen to be so small that
  $a_{\beta}-\eps_{\beta}>d-3+\eps$.
  Consider the walks with indices in~$K_{\beta}$; after hitting
  $V_{\beta}$ the rest of the trajectory is conditionally
  independent from the initial part, so
  Lemma~\ref{l_diam} and~\eqref{cond_behaviour} imply that,
  for $j\in K_{\beta}$
  \[
    P\Big[\big\{X^{(j)}_{H^j_{V_{\beta}}+1},
        \ldots,X^{(j)}_{H^j_{V_{\beta}}+n^2}\big\}
      \cap A\neq \emptyset\Big] \geq
    \begin{cases}
      \displaystyle
      \frac{cn^{-\eps}}{\ln n}, & d=3,\\[0.5cm]
      \displaystyle\frac{cn^{-\eps}}{n^{d-3}}, & d\geq 4,
    \end{cases}
  \]
  for any connected set
  $A\subset B(2n)$ such that $\diam(A)\geq n^{1-\eps}$.
  Using this together with~\eqref{control_Lm} and~\eqref{diam>},
  we conclude the proof of Proposition~\ref{p_connectedness}.
\end{proof}

\section{Proof of Theorem~\ref{t_connected}} 
\label{s_connected}

Using Proposition~\ref{p_connectedness}, it is straightforward to show
Theorem~\ref{t_connected}.
Denote by $\eta^{(n)}_u$ the number
of trajectories at level $u$ entering $B(n)$, that is the number of
trajectories in the support $\mu^u_{B(n)}$ (recall \eqref{e:muKu} for
  the notation). By the definition of random interlacement,
$\eta^{(n)}_u$ has Poisson distribution
with parameter $u\capa(B(n)) = \Theta(un^{d-2})$. 
Therefore,
using e.g.\ Chernoff bounds we obtain for small enough~$c_1$ 
and large enough $c_2$ that
\begin{equation}
  \label{est_num_eta}
  \IP[c_1u n^{d-2}<\eta^{(n)}_u < c_2un^{d-2}]
  \geq 1-\se(n).
\end{equation}

Assume for the moment that
\begin{equation}
  \label{int_absurd}
  \IP[\text{there exists $u>0$ such that $\I^u$ is not connected}]>0.
\end{equation}
Then, one can find $\eps>0$ and $0<u<{\hat u}<\infty$ such that
\[
  \IP[\text{there exists $u'\in[u,{\hat u}]$
      such that $\I^{u'}$ is not connected}] \geq\eps,
\]
and so, denoting by
$\I^u_{(n)} := \bigcup_{i:w_i\in \supp \mu^u_{B(n)}}\Ran w_i$ the interlacement set generated
by the trajectories that intersect~$B(n)$, we have
\[
  \liminf_{n\to\infty} \IP[\text{there exists $u'\in[u,{\hat u}]$
      such that $\I^{u'}_{(n)}$ is not connected}] \geq\eps.
\]
This, however, contradicts Proposition~\ref{p_connectedness}:
putting
$\heta^{(n)}_1=\eta^{(n)}_u$,
$\heta^{(n)}_2=\eta^{(n)}_{{\hat u}}$,
and using~\eqref{est_num_eta}, we see that,
if both events in the left-hand sides of~\eqref{n2connected}
and~\eqref{connected_lower_level} occur, then $\I^{u'}_{(n)}$
should be connected for all $u'\in[u,{\hat u}]$;
on the other hand, the probability of these events
approaches~$1$ as $n\to\infty$.
So,~\eqref{int_absurd} cannot be true.
\qed

\section{Large deviations for the internal distance} 
\label{s:ap}

In this section we prove Theorem~\ref{t_ap}. To this end we fix
$a\in (0,1/3)$ and
investigate the properties of  $\mathcal I^u$ when restricted to
\[
  G^{(n)}_a = \bigcup_{k=0}^n B(k\e_1,n^a).
\]
In words, $G^{(n)}_a$ is the $n^a$-neighbourhood of the segment
between the origin and $n\e_1$
(recall that $B(x,r)$ denotes the ball in the $\|\cdot\|_\infty$-distance).

First,
we need the following elementary estimate on $e_{G_a^{(n)}}(A)$.
\begin{lemma}
  \label{l:escape_sausage}
  Let $F_k$ be the hyperplane $\{x\in\Z^d:x\cdot \e_1 =k\}$. Then, for
  any $k\in \{-\lf n^a \rf + 1, \dots, \lf n + n^a\rf-1\}$, it holds that
  \begin{equation*}
    e_{G^{(n)}_a}(F_k)=
    \sum_{x\in \partial G^{(n)}_a\cap F_k}
    P_x[\tilde H_{G^{(n)}_a}=\infty] \geq
    \begin{cases}
      \displaystyle
      \frac{cn^{a(d-3)}}{\ln n}, & d=3,\\[0.3cm]
      \displaystyle
      {c}{n^{a(d-3)}}, & d\geq 4.
    \end{cases}
  \end{equation*}
\end{lemma}
\begin{proof}
  We adapt the proof of Proposition~2.4.5 of~\cite{Law91}.  Let
  $\tilde G_{n}=([-4n,4n]\times[-n^a,n^a]^{d-1})\cap \mathbb Z^d$.
  As $\tilde G_n -\ell \e_1  \supset G_a^{(n)}-k \e_1$ for any $k$ as in the
  statement and $\ell\in \{-2n,\dots,2n\}$,
  \begin{equation}
    \label{e:fga}
    \sum_{z\in \partial G^{(n)}_a\cap F_k}
    P_z[\tilde H_{G^{(n)}_a}=\infty]\ge
    \sum_{z\in \partial \tilde G_n\cap F_\ell}
    P_z[\tilde H_{\tilde G_n} = \infty].
  \end{equation}
  Let $W=\tilde G_n\cap \{x:|x\cdot \e_1|> 2n\}$. It is elementary to see that
  $P_0[H_W=\infty]\ge c$.
  Inspecting the proof of \cite[Proposition 2.4.1(c)]{Law91}, denoting by
  $L$ the last time the random walk visits $\tilde G_n$, we get
  \begin{equation*}
    \begin{split}
      c&\le P_0[H_W=\infty] \\ &\le P_0[|X_L\cdot \e_1|\le  2n]
      \\&=
      \sum_{z\in\partial \tilde G_n,|z\cdot \e_1|\le 2n}g(0,z)
      P_z[\tilde H_{\tilde G_n}=\infty]
      \\&\le \sum_{\ell=-2n}^{2n}\sup_{z\in \partial \tilde G_n\cap F_\ell}
      g(0,z) \sum_{z\in \partial \tilde G_n\cap F_\ell}
      P_z[\tilde H_{\tilde G_n}=\infty] .
    \end{split}
  \end{equation*}
  Using \eqref{e:fga}, we then get
  \begin{equation*}
    c\le
    \sum_{z\in \partial G_a^{(n)}\cap F_k}
    P_z[\tilde H_{\tilde G_n}=\infty]
    \sum_{\ell=-2n}^{2n}\sup_{z\in \partial \tilde G_n\cap F_\ell} g(0,z) .
  \end{equation*}
  The lemma then follows using the elementary asymptotics
  $g(0,z)\asymp \|z\|^{2-d}$
  (recall~\eqref{Green_asymp>}--\eqref{Green_asymp<}).
\end{proof}

Let $\eta_n=\eta_n(u,a)$ be the number of trajectories of
$\mu^u_{G^{(n)}_a}$ (recall~\eqref{e:muKu}). As before, we enumerate the
corresponding random walks as $X^{(1)},\ldots, X^{(\eta_n)}$, denote
their starting positions by $x^{(1)},\dots, x^{(\eta_n)}$, and let
$R_k(m)$ be the set of different sites visited by $k$th random walk by
time~$m$.

Let us define
\begin{equation}
  \label{def_Uj}
  U_k = B(\lf kn^a\rf\e_1,n^a) =
  \{y\in G^{(n)}_a : |y\cdot\e_1-\lf kn^a\rf|<n^a\},
  \qquad k=0,\dots, n.
\end{equation}
Due to Lemma~\ref{l:escape_sausage}, for any $k\in \{0,\dots,n\}$,
\begin{equation*}
  e_{G_a^{(n)}}(U_k) \ge f_d(n) :=
  \begin{cases}
    \displaystyle
    \frac{cn^{a(d-2)}}{\ln n}, & d=3,\\[0.3cm]
    \displaystyle
    {c}{n^{a(d-2)}}, & d\geq 4.
  \end{cases}
\end{equation*}

Let $\eta_{n,k}=|\{i\le \eta_n:X^{(i)}_0\in U_k \}|$ be the number
of walks starting in $U_k$.
Using the large deviation properties of the Poisson distribution,
as in~\eqref{est_num_eta}, we obtain
\begin{equation}
  \label{e:entracesaussage}
  \mathbb P\big[c_1f_d(n)\le \eta_{n,k} \le
    c_2f_d(n)\big]\ge 1-\se(n).
\end{equation}

We now fix a small positive constant $\eps>0$, and define for
$1\le k\le \eta_n$
\begin{equation*}
  \begin{split}
    j_k&=\inf\{j\ge 0:
      X^{(k)}_{n^{2(a+\varepsilon )}+(j-1) n^{2a}+i} \notin G_a^{(n)},
      \forall i=0,\dots, n^{2a}\},\\
    \hat t_k &= n^{2(a+\varepsilon )}+j_k n^{2a},\\
    \tilde t_k &= \inf\{j\ge 0: \diam (R_k(j))\ge 2 n^{2(a+\varepsilon) }\}
    \wedge \hat t_k.
  \end{split}
\end{equation*}
Denote by $\hat \I$ (respectively, $\tilde \I$) the `interlacement' set
formed only by the initial pieces of length~$\hat t_k$ (respectively,
  $\tilde t_k$) of the trajectories $(X^{(k)}, k=1,\ldots,\eta_n)$:
\[
  \hat{\I}  = \bigcup_{k=1}^{\eta_n} R_k(\hat t_k),\qquad
  \tilde{\I}  = \bigcup_{k=1}^{\eta_n} R_k(\tilde t_k).
\]
Observe that $\I^u\supset \hat \I\supset \tilde \I$. Further, by the central limit
theorem, for any $x\in G_a^{(n)}$, $P_x[T_{G_a^{(n)}}\le n^{2a}]\ge c$.
Therefore, using the strong Markov property recursively on the definition
of $j_k$,
\begin{equation*}
  \mathbb P[j_k\ge n^\varepsilon ]\le \se (n).
\end{equation*}
When $j_k\le n^\varepsilon $, then
$\diam(R_k(\hat t_k))\le 2n^{2(a+\varepsilon )}$. Therefore,
\begin{equation}
  \label{e:hatItildeI}
  \mathbb P[\hat \I \neq \tilde \I]\le \se(n).
\end{equation}

Heuristically, the set $\hat \I$, is `well suited' for application of
Proposition~\ref{p_connectedness} as it has no `dangling ends' in
$G_a^{(n)}$. By this we mean that knowing that
$X^{(k)}$ is in $G_a^{(n)}$ at some time $j$, its next $n^{2a}$ steps will
be contained in $\hat \I$:
\begin{equation*}
  \{X^{(k)}_j\in G_a^{(n)}, j\le \hat t_k\}   \implies \{\hat t_k \ge j+n^{2a}\}.
\end{equation*}
On the other hand, the trajectories in $\tilde \I$ are `short range', which
will introduce some independence later.

We now introduce a notation that will be useful many times,
see Figure~\ref{f:psis} for its illustration.
For $x\in\Z^d\setminus\{0\}$ and $y\in\Z^d$, we define
\begin{align*}
  \zeta^{(x)}_0(y) &= \max\{m\leq 0 : mx+y\in\I^u\},\\
  \zeta^{(x)}_{k+1}(y) &= \min\{m>\zeta^{(x)}_k(y) : mx+y\in\I^u\},
  \quad k\geq 0.
\end{align*}
We set
$\psi^{(x)}_k(y):=y+\big(\zeta^{(x)}_k(y)\big)x$ to be the
site on~$\I^u$ corresponding to~$\zeta^{(x)}_k(y)$.
When $y=0$ and/or $x=\e_1$, we omit them from the notation, that is
e.g.~$\zeta_k:=\zeta^{(\e_1)}_k(0)$.

\begin{figure}[ht]
  \centering
  \includegraphics{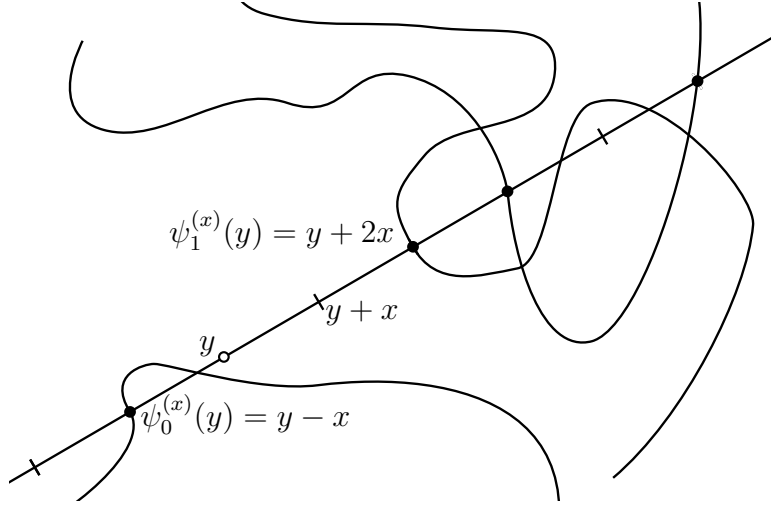}
  \caption{Construction of $\zeta^{(x)}_i(y)$ and $\psi^{(x)}_i(y)$.
    Here ${\zeta^{(x)}_0(y)=-1}$, $\zeta^{(x)}_1(y)=2$,
    $\zeta^{(x)}_2(y)=3$ and $\zeta^{(x)}_3(y)=5$.}
  \label{f:psis}
\end{figure}

As the first step of the proof of Theorem~\ref{t_ap}, we control the
distance between the left- and right-most intersection of $\mathcal I^u$
with the segment $(0,n\e_1]$. More precisely, we want to show that
\begin{equation}
  \label{e:rho_onedir}
  \IP[\rho_u(\psi_1, \psi_0(n\e_1))\geq cn]\le \se(n)
\end{equation}
for a $c$ large. (Observe that $\IP^u_0[\psi_0=0]=1$.)

We write $\hat\rho(x,y)$ for the internal distance of $x$, $y$ on
$\hat \I$ and define $\hpsi_i^{(x)}(y)$ similarly as $\psi_i^{(x)}(y)$,
using $\hat\I$ instead of $\I^u$.  Observe that $\hpsi_i^{(x)}(y)$ depends
on $n$. It may happen that
$\hpsi_1 = \psi_1$, $\hpsi_0(n\e_1)=\psi_0(n\e_1)$, but that is not
certain. In any case,
\begin{equation}
  \label{e:triangle}
  \rho_u(\psi_1, \psi_0(n\e_1)) \le
  \rho_u(\psi_1,\hpsi_1) + \hat\rho(\hpsi_1,\hpsi_0(n\e_1)) +
  \rho_u(\hpsi_0(n\e_1), \psi_0(n\e_1)).
\end{equation}

To bound the right-hand side,
we need few lemmas.
\begin{lemma}
  \label{l:dist_exp_tail}
  Let $g_d(k)=e^{-c_4k}$ when $d\ge 4$, and $g_3(k)=e^{-c_4k/\ln k}$.
  Then, for every $a\in (0,1/3)$, $\varepsilon >0$,
  $x\in \mathbb Z^d\setminus\{0\}$, $y\in\mathbb Z^d$,
  \begin{equation}
    \label{e:inbox}
    \IP\big[\|\psi^{(x)}_1(y)\|_\infty\ge k\big]\le g_d(k),\\
  \end{equation}
  Further, for $y$ such that $B(y,n^{a/2})\subset G_a^{(n)}$,
  \begin{equation*}
    \IP\big[\|\hpsi_1^{(\e_1)}(y)\|_\infty\ge k\big]\le g_d(k)+\se(n).
  \end{equation*}
\end{lemma}

\begin{proof}
  Let $S_k=\{y+jx:0\le j\le k\}$. The first claim follows directly from
  the definition of~$\I^u$ (observe that for any $A\subset\Z^d$
it holds that $\IP^u[A\cap \I^u = \emptyset]=e^{-u\capa(A)}$)
and the simple estimate on the capacity of 
the `segment' $S_k$ (see e.g.~\cite{Law91}, Proposition~2.4.5)
  \begin{equation*}
    \capa S_k =
    \begin{cases}
      \Theta(k/\ln k), \quad & d=3,\\
      \Theta(k), &d\ge 4.
    \end{cases}
  \end{equation*}

  For the second statement, we assume without loss of generality that
  $y=0$, and define
  $A_n=\{0\le k \le n^\varepsilon, k \text{ even}\}$.
  For every $j\in A_n$,
  and $x\in U_j$, by Lemma~\ref{l_diam},
  \begin{equation*}
    q_x(S_k;n^{2(a +\varepsilon) }) \ge
    \begin{cases}
      \displaystyle
      \frac{ck}{((j+1)n^a)^{d-2}\ln k}, & d=3,\\[0.3cm]
      \displaystyle
      \frac{ck}{((j+1)n^a)^{d-2}}, & d\ge4.
    \end{cases}
  \end{equation*}
  Combining this estimate with \eqref{e:entracesaussage}, using
  $\hat t_k\ge n^{2(a+\varepsilon )}$, we obtain in
  $d=3$,
  \begin{equation*}
    \begin{split}
      \IP\big[\|\hpsi_1\|_\infty\ge k\big]
      &\le \se(n)+\prod_{j\in A_n}
      \Big(1-\frac{ck}{((j+1)n^a)^{d-2}\ln k}\Big)^{c_1 n^{a(d-2)}/\ln n}
      \\&\le \se(n)+
      \exp\bigg\{-\frac{c_3 k}{\ln n\ln k}
        \sum_{j\in A_n}\frac{1}{(j+1)^{d-2}}\bigg\}
      \\&\le e^{-c_4 k /\ln k} + \se (n).
    \end{split}
  \end{equation*}
  For $d\ge 4$ the calculation is very similar. Actually, it is
  sufficient to consider only the term $j=0$, as there are no
  logarithmic terms in the denominator.
\end{proof}

As a consequence of the last lemma and Proposition~\ref{p_connectedness}
we obtain,
\begin{align}
  \IP[\rho_u(\hpsi_1,\psi_1)\leq c'n^{2a}]
  &\geq 1-\se(n),
  \label{distance_psi0}\\
  \IP[\rho_u(\hpsi_0(n\e_1),
      \psi_0(n\e_1))\leq c'n^{2a}]
  &\geq 1-\se(n).
  \label{distance_psin}
\end{align}
Indeed, obviously
\begin{equation}
  \begin{split}
    \label{e:odhad}
    \IP&[\rho_u(\psi_1,\hpsi_1)\geq c'n^{2a}]
    \\&\le
    \mathbb P[\{\psi_1,\hpsi_1\}\not\subset B(n^a)]
    +
    \mathbb P\big[\{\psi_1,\hpsi_1\}\subset B(n^a), \rho_u(\psi_1,\hpsi_1)\ge
      c'n^{2a}\big].
  \end{split}
\end{equation}
The first term in the right-hand side is $\se(n)$ by
Lemma~\ref{l:dist_exp_tail}. For the second one, it suffices to set
$c'=3(2\beta (h,d)+1)$ and apply Proposition~\ref{p_connectedness}
to~$U_0$ (recall~\eqref{def_Uj}) and $\heta_1^{(n)}=\heta_2^{(n)}=\eta_{n,0}$
which satisfy the assumptions~\eqref{heta_1} and~\eqref{heta_2} due
to~\eqref{e:entracesaussage}. Claim~\eqref{distance_psi0} then follows.
The proof of~\eqref{distance_psin} is completely analogous.

Similarly, applying Proposition~\ref{p_connectedness} to
the sequence of sets~$U_k$, $k=0,\ldots,n$,
we obtain that
\begin{equation*}
  \IP[\hat\I \text{ is connected}]\geq 1-\se(n).
\end{equation*}

To bound the middle term on the right-hand side of \eqref{e:triangle}, we
consider the sequence of random variables
\begin{equation*}
  \hat T^n_k = \begin{cases}
    \hat \rho(k\e_1,\hpsi_1(k\e_1)), & \text{ if }
    k\e_1\in{\hat\I}, \|\hpsi_1(k\e_1)-k\e_1\|\le n^a\\
    0, & \text{ otherwise }.
  \end{cases}
\end{equation*}
It is clear that on the event
$\{\|\hpsi_1(k\e_1)-k\e_1\|\le n^a:k=0,\dots,n\}$, which by
Lemma~\ref{l:dist_exp_tail} has probability $1-\se(n)$, we have
\begin{equation}
  \label{sum_T}
  \hat \rho(\hpsi_1,\hpsi_0(n\e_1))
  \leq \sum_{k=0}^n \hat T^n_k .
\end{equation}
To control the sum, we need a tail estimate on $\hat T^n_k$ that is
uniform in $n$.
\begin{lemma} For every $k=0,\dots,n$, uniformly in $k$, we have
  \begin{equation*}
    \mathbb P[\hat T^n_k\ge \ell]\le \se(n)+\se(\ell).
  \end{equation*}
\end{lemma}
\begin{proof}
  Without loss of generality we consider $k=0$ only. First, we fix
  $m\le n^a/2$ and control the number of trajectories entering $B(m)$. We
  claim that
  \begin{equation}
    \label{e:entrancesBm}
    \mathbb P[|\{i\le \eta_{n}:H^i_{B(m)}\le \hat t_i\}\ge c_5 m^{d-2}]
    \ge 1-\se (m).
  \end{equation}
  To prove \eqref{e:entrancesBm} we use an argument similar to the proof
  of Lemma~\ref{l:dist_exp_tail}.
  We define
  $A_n=\{0\le k \le n^{\varepsilon/2}, k \text{ even}\}$.
  For every $j\in A_n$,
  and $x\in U_j$, by Lemma~\ref{l_ball},
  \begin{equation*}
    q_x(B(m);n^{2(a +\varepsilon) }) \ge
    \frac{cm^{d-2}}{((j+1)n^a)^{d-2}}.
  \end{equation*}
  Using \eqref{e:entracesaussage}, the number of walks starting in~$U_j$
  hitting $B(m)$ has a Poisson distribution with parameter at least
  \begin{equation*}
    \frac{cf_d(n)m^{d-2}}{((j+1)n^a)^{d-2}} \ge
    \begin{cases}
      \displaystyle
      \frac{c'm^{d-2}}{(j+1)^{d-2} \ln n}, \qquad &d=3,\\[3mm]
      \displaystyle
      \frac{c'm^{d-2}}{(j+1)^{d-2} },  &d\ge 4.
    \end{cases}
  \end{equation*}
  Using the stability of the Poisson distribution, this yields that the
  number of walks starting in $\bigcup_{j\in A_n}U_j$ hitting $B(m)$
  has a Poisson distribution with mean at least $c'' m^{d-2}$.
  Claim~\eqref{e:entrancesBm} then follows from the large deviation
  properties of the Poisson distribution again.

  We now apply Proposition~\ref{p_connectedness} with $m$ instead of $n$,
  $G_m=B(m)$ and $\heta_1^{(m)}$, $\heta_2^{(m)}$ being the number of
  walks entering $B(m)$. Assumptions \eqref{heta_1}, \eqref{heta_2} are
  satisfied by the previous discussion. The construction of $\hat \I$
  assures that the walks do not stop earlier than after making
  $2m^2\le 2n^{2a}$ steps. Therefore, by an argument similar to proof of
  \eqref{distance_psi0}, for $c'=2(2 \beta (h,d)+3)$
  \begin{equation*}
    \mathbb P[\hat T_0^n\ge c'm^2]\le \mathbb P[0\in \hat \I]\big(
      \mathbb P_0^u[\hpsi_1 \notin B(m)] +
      \mathbb P_0^u[\hat T_0^n\ge c'm^2,\hpsi_1\in B(m)]\big).
  \end{equation*}
  Both terms in the parentheses are $\se(m)$, the first one by
  Lemma~\ref{l:dist_exp_tail}, the second one by
  Proposition~\ref{p_connectedness}. Taking $m_\ell$ such that
  $c'm_\ell^2 = \ell$, the lemma follows for for $\ell < c'(n/2)^{2a}$.

  For the remaining $\ell$'s it suffices to observe that
  $\se(\ell)\le \se(n)$ and apply the same reasoning as before with
  $B(n^a)$ instead of $B(m)$.
\end{proof}

We can now control the sum \eqref{sum_T}. To this end we define
$\tilde \rho $, $\tilde \psi_i^{(x)}(y)$, $\tilde T_k^n$ in the same way
as $\hat \rho$, $\hat \psi_i^{(x)}(y)$, $\hat T_k^n$, using $\tilde \I$
instead of $\hat\I$. Due to \eqref{e:hatItildeI},
\begin{equation}
  \label{e:tildeThatT}
  \mathbb P[\text{there exists }
    0\le k \le n:\hat T_n^k \neq \tilde T_n^k] \le \se(n).
\end{equation}
Therefore, by the previous lemma, for every $k$,
\begin{equation*}
  \mathbb P[\tilde T^n_k\ge \ell]\le \se(n)+\se(\ell).
\end{equation*}
The random variable $\tilde T^n_k$ depends only on the random walks that
can enter the ball $B(k\e_1, n^a)$. As
$\diam(R_i(\tilde t_k))\le 2n^{2(\alpha +\varepsilon )}$ by definition,
setting $b_n=5n^{2(\alpha +\varepsilon )}$, this implies that for every
$j\in \{1,\dots, b_n\}$ the random variables
$(T_{kb_n +j}:k=1,\ldots, n/b_n)$ are independent.

Therefore, for large enough~$c$, using  the
observation above \eqref{sum_T}, and \eqref{e:tildeThatT},
\begin{equation}
  \begin{split}
    \label{e:middlebound}
    \IP[\hat \rho(\hpsi_1,\hpsi_0(n\e_1))\geq cn/2]
    &\leq  \IP\Big[\sum_{k=0}^n \hat T_k\geq cn/2\Big] + \se(n)\\
    &\leq  \IP\Big[\sum_{k=0}^n \tilde T_k\geq cn/2\Big] + \se(n)\\
    &\leq \IP\Big[\exists  j\in[1,b_n] :
      \sum_{k=0}^{n/b_n} T_{kb_n+j}\geq cn/(2b_n) \Big]
    \\
    &\leq 1-\se(n),
  \end{split}
\end{equation}
where for the last inequality we used the fact that $a<\frac{1}{3}$ and
then applied a large deviation bound for random variables without
exponential moments (e.g., Theorem~1.1 of~\cite{Nag79}).

Combining \eqref{e:middlebound}, \eqref{distance_psi0},
\eqref{distance_psin} with \eqref{e:triangle}, the inequality
\eqref{e:rho_onedir} follows.

To conclude the proof of Theorem~\ref{t_ap}, observe
that~\eqref{e:rho_onedir} implies that for large enough~$c$
\begin{equation*}
  \begin{split}
    \IP[\exists k\in[n/2&,n] \text{ such that }k\e_1\in\I^u \text{ and }
      \rho_u(\psi_1, k\e_1)\geq cn]
    \\
    &\leq \IP\Big[\bigcup_{j\in[n/2,n]}
      \{\rho_u(\psi_1, \psi_0(j\e_1))\geq cn\}\Big]\\
    & \leq\se(n),
  \end{split}
\end{equation*}
and then, since for any $k\in[0,n/2]$ such that $k\e_1\in\I^u$ one can
write
$\rho_u(\psi_1, k\e_1)\leq \rho_u(\psi_1, \psi_0(n\e_1))
+\rho_u(\psi_0(n\e_1), k\e_1)$,
we have
\begin{equation}
  \label{reach_line}
  \IP[\exists k\in[-n,n] \text{ such that }k\e_1\in\I^u \text{ and }
    \rho_u(\psi_1, k\e_1)\geq cn] \leq\se(n).
\end{equation}
Observe that, by symmetry, \eqref{reach_line} also holds
if one substitutes~$\e_1$ by any coordinate vector~$\e_j$,
$j=2,\ldots,d$.
The claim of the lemma then follows if one writes, on
$x=(x_1,\ldots,x_d)\in \I_u$ and $0\in \I^u$,
\begin{equation*}
  \begin{split}
    \rho_u(0,x)&\leq \rho_u(0,\psi_1^{(\e_1)}(0))
    +\rho_u(\psi_1^{(\e_1)(0)},\psi_0^{(\e_1)}(x_1\e_1))
    \\&
    +\rho_u(\psi_0^{(\e_1)}(x_1\e_1),\psi_1^{(\e_2)}x_1\e_1)
    +\rho_u(\psi_1^{(\e_2)}x_1\e_1,\psi_0^{(\e_2)}(x_1\e_1+x_2\e_2))
    \\&+ \dots
    \\&+
    \rho_u(\psi_0^{(\e_{d-1})}(x_1\e_1+\cdots+x_{d-1}\e_{d-1})
      ,\psi_1^{(\e_d)}(x_1\e_1+\cdots+x_{d-1}\e_{d-1}))
    \\&\qquad +\rho_u(\psi_1^{(\e_d)}(x_1\e_1+\cdots+x_{d-1}\e_{d-1}),\psi_0^{(\e_d)}(x)),
  \end{split}
\end{equation*}
and uses the same reasoning for the even terms in the right-hand side,
and bounds the odd terms using~\eqref{e:inbox} and
Proposition~\ref{p_connectedness}, using the same argument as e.g.\
in~\eqref{e:odhad}. This completes the proof of Theorem~\ref{t_ap}. \qed

\section{Proof of the shape theorem} 
\label{s_proof_shape}
In this section, to prove Theorem~\ref{t_shape}, we use more or less
standard argument based on the Subadditive Ergodic Theorem. For reader's
convenience, let us state this theorem here (we use the version
  of~\cite{Lig85}):
\begin{theorem}
  \label{t_subadd}
  Suppose that $\{Y(m,n)\}$ is a collection of
  positive random variables indexed by integers
  satisfying $ 0 \leq m< n $ such that
  \begin{itemize}
    \item[(i)] $ Y(0,n) \le Y(0,m) + Y(m,n)$
    for all $ 0 \le m < n $;
    \item[(ii)] The joint distribution of $ \{Y(m+1, m+k+1), k \ge 1\} $
    is the same as that of $ \{ Y(m, m+k), k \ge 1\} $ for each $ m
    \ge 0 $;
    \item[(iii)] For each $ k \ge 1 $ the sequence of random variables
    $ \{Y(nk, (n+1)k), n \ge 1\} $ is a stationary
    ergodic process;
    \item[(iv)] $ \IE Y(0,1) < \infty $.
  \end{itemize}
  Then, it holds that
  \[
    \lim_{n \to \infty} \frac{Y(0,n)}{n} = \inf_{n \ge 0}
    \frac{\IE Y(0,n)}{n} \qquad \text{a.s.}
  \]
\end{theorem}

We are going to verify the hypotheses of Theorem~\ref{t_subadd}
for the sequence of random variables
\begin{equation*}
  Y(m,n) = \rho_u(\psi^{(x)}_m,\psi^{(x)}_n),
\end{equation*}
under the measure~$\IP_0^u$. First, (i) is obvious since $\rho_u$
is a metric. Stationarity and ergodicity in (ii)--(iii) follow
from the corresponding properties of $\I^u$, see Theorem~2.1
of~\cite{Szn10}.
The property~(iv) then follows from the estimate
\begin{equation*}
  \IP_0^u[\rho_u(0,\psi^{(x)}_1)>n]
  \leq \se(n),
\end{equation*}
which can be proved by applying the same procedure as in \eqref{e:odhad},
using Lemma~\ref{l:dist_exp_tail} and Proposition~\ref{p_connectedness}.

Theorem~\ref{t_subadd} implies that for
any~$x\in\Z^d$ there exists a positive
number $\sigma_u'(x)$ such that
\begin{equation}
  \label{def_sigma'}
  \IP^u_0\Big[\lim_{n\to\infty}
    \frac{\rho_u(0,\psi^{(x)}_n)}{n}=\sigma_u'(x)\Big] = 1.
\end{equation}
Then, define for $x\neq 0$
\begin{equation}
  \label{def_sigma}
  \sigma_u(x) = \frac{\sigma_u'(x)}{\IE^u_0 \zeta^{(x)}_1},
\end{equation}
and $\sigma_u(0):=0$.
With~\eqref{def_sigma'} it is straightforward to obtain
(observe that, according to our notations,
  $\psi^{(x)}_0(nx)$ is either~$nx$ itself in the case~$nx\in\I^u$,
  or it is the `last site before $nx$' on the
  discrete ray $\{kx, k\geq 0\}$ if~$nx\notin\I^u$), using also
the usual Ergodic Theorem and \eqref{e:inbox}, that
\begin{equation}
  \label{lim_sigma}
  \lim_{n\to\infty} \frac{\rho_u(0,\psi^{(x)}_0(nx))}{n}
  =\sigma_u(x) \qquad \text{$\IP^u_0$-a.s.}
\end{equation}
It is also straightforward to obtain that for any integer~$m$
and $x\in\Z^d$, it holds that $\sigma_u(mx)=m\sigma_u(x)$; this permits
us to extend~$\sigma_u$ to~$\Q^d$ by $\sigma_u(x):=m^{-1}\sigma_u(mx)$,
where~$m$ is such that $mx\in\Z^d$.
Also, it is clear that $\sigma_u(x)\geq \|x\|_1$ for any $x\in\Q^d$.

Next, the goal is to prove that~$\sigma_u$ is a norm.
\begin{lemma}
  \label{l_norm_rationals}
  For all $x,y\in\Q^d$ we have
  \begin{equation}
    \label{norm_rationals}
    \sigma_u(x+y)\leq \sigma_u(x)+\sigma_u(y).
  \end{equation}
\end{lemma}

\begin{proof}
  Abbreviate $b_x=\IE^u_0\zeta^{(x)}_1$; from the Ergodic Theorem
  we obtain
  \begin{equation}
    \label{norm_psi}
    \lim_{n\to\infty}\frac{\|\psi^{(x)}_{b_x^{-1}n}-nx\|}{n} = 0
    \qquad \text{ $\IP^u_0$-a.s.}
  \end{equation}
  Since~$\rho_u$ is a metric, we have (see Figure~\ref{f_norm})
  \begin{figure}
    \centering
    \includegraphics{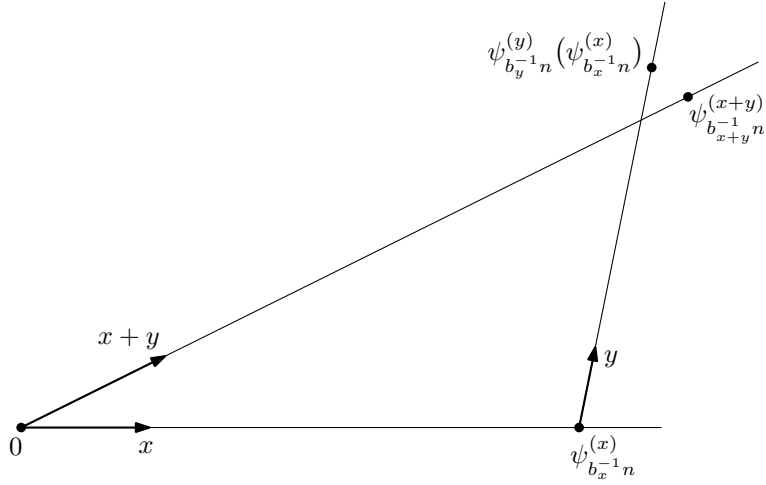}
    \caption{On the proof of Lemma~\ref{l_norm_rationals}}
    \label{f_norm}
  \end{figure}
  \begin{equation}
    \label{triangle}
    \frac{\rho_u\big(0,
        \psi^{(y)}_{b_y^{-1}n}(\psi^{(x)}_{b_x^{-1}n})\big)}{n}
    \leq \frac{\rho_u(0,\psi^{(x)}_{b_x^{-1}n})}{n}
    + \frac{\rho_u\big(\psi^{(x)}_{b_x^{-1}n},
        \psi^{(y)}_{b_y^{-1}n}(\psi^{(x)}_{b_x^{-1}n}) \big)}{n}
  \end{equation}
  for any~$n$.
  Then, the trick is to take the limit as $n\to\infty$
  in~\eqref{triangle} \emph{in probability}. First of all,
  a direct application of~\eqref{def_sigma'}--\eqref{def_sigma}
  shows that
  the first term in the right-hand side of~\eqref{triangle}
  converges to~$\sigma_u(x)$, even $\IP^u_0$-a.s. Next,
  under $\IP^u_0$ it holds that
  $\rho_u\big(\psi^{(x)}_{b_x^{-1}n},
    \psi^{(y)}_{b_y^{-1}n}(\psi^{(x)}_{b_x^{-1}n}) \big)$ is equal to
  $\rho_u\big(0,\psi^{(y)}_{b_y^{-1}n}\big)$ in distribution,
  so the second term in the right-hand side of~\eqref{triangle}
  converges to~$\sigma_u(y)$ in distribution and hence in probability.
  As for the term in the left-hand side of~\eqref{triangle},
  write
  \begin{equation}
    \label{oc_1st_term}
    \frac{\rho_u\big(0,
        \psi^{(y)}_{b_y^{-1}n}(\psi^{(x)}_{b_x^{-1}n})\big)}{n}
    \geq \frac{\rho_u\big(0,
        \psi^{(x+y)}_{b_{x+y}^{-1}n}\big)}{n} -
    \frac{\rho_u\big(\psi^{(y)}_{b_y^{-1}n}(\psi^{(x)}_{b_x^{-1}n}),
        \psi^{(x+y)}_{b_{x+y}^{-1}n}\big)}{n}.
  \end{equation}
  Again, the first term in the right-hand side of~\eqref{oc_1st_term}
  converges $\IP^u_0$-a.s.\ to~$\sigma_u(x+y)$. To obtain
  that the second term in the right-hand side of~\eqref{oc_1st_term}
  converges to~$0$ in probability, observe that
  \begin{align}
    \frac{\|\psi^{(y)}_{b_y^{-1}n}(\psi^{(x)}_{b_x^{-1}n})-
      \psi^{(x+y)}_{b_{x+y}^{-1}n}\|}{n}
    &\leq \frac{\|(x+y)n-
      \psi^{(x+y)}_{b_{x+y}^{-1}n}\|}{n} +
    \frac{\|nx-\psi^{(x)}_{b_{x}^{-1}n}\|}{n} \nonumber\\
    &\qquad + \frac{\|\psi^{(y)}_{b_y^{-1}n}
      (\psi^{(x)}_{b_x^{-1}n})-\psi^{(x)}_{b_{x}^{-1}n}-ny\|}{n}.
    \label{3norms}
  \end{align}
  Since the third term in the right-hand side of~\eqref{3norms} equals in
  distribution to $n^{-1}\|ny-\psi^{(y)}_{b_{y}^{-1}n}\|$,
  \eqref{norm_psi} implies that the left-hand side of~\eqref{3norms}
  converges to~$0$ in probability, and so Theorem~\ref{t_ap} implies that
  the second term in the right-hand side of~\eqref{oc_1st_term} converges
  to~$0$ in probability. This proves~\eqref{norm_rationals}.
\end{proof}

Now, Lemma~\ref{l_norm_rationals} shows that~$\sigma_u$ can be extended
to a norm in~$\R^d$ by continuity,
and we are able to finish the proof of Theorem~\ref{t_shape}.

\begin{proof}[Proof of Theorem~\ref{t_shape}]
  At this point the argument is quite standard.
  Let
  \[
    D_u = \{x \in \R^d : \sigma_u(x) \le 1 \}.
  \]

  Let $\eps'=(1-\eps)^{-1}-1$, and $\eps''=1-(1+\eps)^{-1}$. To prove
  Theorem~\ref{t_shape}, it is enough to prove that
  $ nD_u\cap\I^u \subset \Lambda^u((1+\eps')n)$ and
  $\Lambda^u((1-\eps'')n) \subset nD_u$ for all~$n$ large enough, $\IP^u_0$-a.s.

  Since~$D_u$ is compact, one can find a finite set
  $ F:= \{x_1, \ldots, x_k\} \subset D_u\cap \Q^d$ such that
  $ \sigma_u(x_i) < 1 $ for $i=1,\ldots,k$, and
  (with $C$ from Theorem~\ref{t_ap})
  \[
    D_u \cap\I^u \subset \bigcup_{i=1}^{k}B(x_i, C^{-1}\eps').
  \]
  Consider any $x_i\in F$; let~$m_i$ be the minimal positive
  integer such that $m_ix_i\in\Z^d$. Let~$n=jm_i+s$, where
  $0\leq s \leq m_i-1$. Then, for all~$n$ large enough it holds
  by~\eqref{lim_sigma} that
  $\psi^{(m_ix_i)}_0(jm_ix_i)\in \Lambda^u(n)$, $\IP^u_0$-a.s.

  Now, Theorem~\ref{t_ap}, \eqref{e:inbox}
  and the Borel-Cantelli lemma imply
  that $\IP^u_0$-a.s.\
  for all~$n$ large enough we have
  \[
    B( n x_i, C^{-1}n\eps')\cap\I^u
    \subset
    \Lambda^u (jm_ix_i,n\eps'),
  \]
  for all $i=1,2, \dots ,k$.
  So $ nD_u\cap\I^u  \subset \Lambda^u((1+\eps')n)$,
  which completes the first part of the proof.

  For the second, choose $G := \{y_1, \ldots, y_k\} \subset
  (2D_u \setminus D_u)\cap\Q^d$ in such a way that
  \[
    2D_u \setminus D_u
    \subset \bigcup_{i=1}^k B(y_i, \eps''\delta ).
  \]
  Notice that $ \sigma_u(y_i) > 1 $ for $i=1,\ldots,k$. Again,
  $ n G \cap \Lambda^u(n) =\emptyset$ for all~$n$ large enough $\IP^u_0$-a.s.
  Analogously, by Theorem~\ref{t_ap} and Borel-Cantelli we get that for
  all~$n$ large enough, if
  $\Lambda^u((1-\eps'')n) \cap n(2D_u \setminus D_u ) \neq \emptyset$, then
  $\Lambda^u(n) \cap n G \neq \emptyset$. This shows that
  $\Lambda^u((1-\eps'')n) \subset nD_u  $ for all~$n$ large enough,
  $\IP^u_0$-a.s., and so concludes the proof of Theorem~\ref{t_shape}.
\end{proof}

\section{Random walk on the torus} 
\label{s:torus}

It remains to show Theorem~\ref{t_torus}. We recall that $\mathbb T_N^d$
denotes the $d$-dimensional discrete torus of size $N$, $P^N$ the law of
the simple random walk on $\mathbb T_N^d$ started from the uniform
distribution, and $\rho^u_N(x,y)$ the internal distance within the set
$\I^u_N$ of
sites visited by the random walk before time $uN^d$,
$\I^u_N=\{X_0,\dots,X_{\lf uN^d\rf}\}$.

Let $B_N(x,r)\subset \mathbb T_N^d$ be the ball of
radius $r$ around $x$ in the usual distance, $d_N$, on the torus.
We first control the internal distance in balls of radius $\ln^\gamma N$.
\begin{lemma}
  \label{l:toruslocal}
  Let $x\in \mathbb T_N^d\subset \mathbb Z^d$.
  Then, for $c_1$, $\gamma $ large enough,
  \begin{equation}
    \label{e:toruslocal}
    P^N\big[\text{there exist }y,z\in B_N(x,\ln^\gamma N)
      \text{ such that }
      \rho_N^u(y,z)\ge c_1 \ln^\gamma N\big] = o(N^{-d}).
  \end{equation}
\end{lemma}

Before proving this lemma, let us explain how it implies
Theorem~\ref{t_torus}.
\begin{proof}[Proof of Theorem~\ref{t_torus}]
  By the lemma and a simple union bound, with
  probability tending to $1$,  the event in \eqref{e:toruslocal} is
  satisfied for all $x\in \mathbb T_N^d$. If this is the case, we can chain
  these boxes to obtain the claim of the theorem. More precisely, consider
  $x,y$ such that $d_n(x,y)> \ln^\gamma N$. Then, one can find points
  $x=x_1,x_2, \ldots, x_n=y$ such that $x_{i+1}\in B(x_i, \ln^\gamma  N)$,
  $i<n$, and $\sum_{i=1}^{n-1} d_N(x_i,x_{i+1})\le 2 d_N(x,y)$. As we
  assume that the event in \eqref{e:toruslocal} is satisfied for all balls,
  for all $i<n$,
  \begin{equation*}
    \rho^u_N(x_i,x_{i+1}) \le c_1 d_N(x_i,x_{i+1}).
  \end{equation*}
  The theorem then follows using the triangular inequality, setting
  $\bar C= 2c_1$.
\end{proof}

\begin{proof}[Proof of Lemma~\ref{l:toruslocal}]
  Let $r=\ln^\gamma N$ and $R=C r$, with~$C$ of
  Theorem~\ref{t_ap}.
  By Theorem~1.1 of~\cite{TW11}, for any $\alpha>0$, there exists a
  coupling~$Q$ of random interlacement on~$\mathbb Z^d$ and random walk
  on the torus, such that
  \begin{equation*}
    Q\big[\big(\I^{u(1-\varepsilon )}\cap B(x,R)\big)
      \subset \big(\I_N^u\cap B_N(x,R)\big)\subset
      \big(\I^{u(1+\varepsilon )}\cap B(x,R)\big)\big]\ge 1- N^{-\alpha }.
  \end{equation*}
  For points that are in $\I^{u(1-\varepsilon )}\cap B(x,r)$ we can use
  Theorem~\ref{t_ap} and obtain the required statement.
  For points in $\I^u_N\setminus \I^{u(1-\varepsilon )}$, however, this
  simple argument fails and we need more details on the coupling
  construction.

  The construction starts by splitting the random walk trajectory into
  so-called excursions. These excursions are independent simple random
  walk trajectories started at the boundary of $B_N(x,R)$ and stopped
  when staying a sufficiently long time out of $B_N(x,N^{1-\varepsilon })$,
  see Section 4 of \cite{TW11} for the precise definition.

  We denote the excursions started before time $uN^d$ by
  $X^{(1)}, \ldots, X^{(\eta )}$, where~$\eta$ is random.
  These excursions are constructed in such a way that
  \begin{equation}
    \label{e:torusinclusion}
    \Big(\bigcup_{i=1}^{\eta -1} \Ran X^{(i)} \cap B(x,R)\Big)
    \subset \big(\I^u_N\cap B(x,R)\big)
    \subset
    \Big(\bigcup_{i=1}^{\eta } \Ran X^{(i)} \cap B(x,R)\Big).
  \end{equation}
  Using Lemma~4.3 of \cite{TW11}, it is easy to see
  that the random variable $\eta $ satisfies
  \begin{equation}
    \label{e:torusnumber}
    P^N[(u+\varepsilon ) \capa B(x,R)
      \ge \eta \ge (u-\varepsilon )\capa B(x,R)]
    \ge 1- e^{-c \capa B(x,R)}.
  \end{equation}
  Further, combining Lemmas~3.9,~3.10 of \cite{TW11}, it follows that the
  distribution of the starting points satisfies for every
  $z\in \partial B(x,R)$ and $i=2,\ldots,R$
  \begin{equation}
    \label{e:torusinitial}
    P^N[X^{(i)}_0=z]\ge
    \frac{(1-\varepsilon )e_{B(x,R)}(z) }{\capa B(x,R) },
  \end{equation}
  that is, it is close to the normalised equilibrium measure. The
  distribution of the first excursion cannot be controlled so precisely,
  as in principle it can start inside of $B(x,R)$, but it is not issue for
  us.

  The proof of Lemma~\ref{l:toruslocal} is then completely analogous to
  proof of Theorem~\ref{t_ap}. It suffices to observe that the only
  property of the random interlacement that we used in the proof of
  Theorem~\ref{t_ap} are the bounds on the number and starting
  distribution of trajectories entering a fixed set. These bounds follow
  from~\eqref{e:torusnumber}, \eqref{e:torusinitial}. Finally one should
  observe that $\se(R)= o( N^{-d})$ if~$\gamma $ is chosen large enough.

  There is a small issue with the fact that the last point
  of the trajectory, $X_{\lf uN^d\rf}$, might be contained in the last excursion,
  cf.~\eqref{e:torusinclusion}.
  To solve
  this issue, observe that our techniques apply to the both `clusters'
  $\mathcal C:=\bigcup_{i=1}^{\eta-1 }\Ran X^{(i)}$ and
  $\bar {\mathcal C}:=\bigcup_{i=1}^{\eta }\Ran X^{(i)}$. Hence, with
  probability $1-\se(R)$, the internal distances on these clusters within
  $B(x,r)$ are bounded by~$c_1 r$. If this is the case, then the trajectory of
  $X^{(\eta) }$ must intersect~$\mathcal C$ at least every $(2c_1+1)$-steps.
  For any $x\in B(x,r)\cap X^{(\eta )}\cap \I^u_N$ there is thus a path of
  length at most $(2c_1+1)$  lying inside of~$\I^u_N$
  which connects~$x$
  to~$\mathcal C$. Lemma~\ref{l:toruslocal} then follows by triangular
  inequality, by increasing $c_1$ to $2(2c_1+1)+c_1$.
\end{proof}

\appendix 
\section{Domination by Bernoulli percolation}
We sketch here a simple argument proving Theorem~\ref{t_ap} in $d\ge 5$,
with~$\delta =1$. This argument is based on the domination of the
interlacement set~$\I^u$ in thick two-dimensional slabs by the standard
Bernoulli percolation. This domination seems to be folklore in the random
interlacement community, but to our knowledge it does not appear in any
previous publications.

Let~$K$ be a sufficiently large constant  and
$\varepsilon \in (0,1)$. Let $E_2$ be the set of nearest-neighbour edges
of $\mathbb Z^2$, and for every $e=(x,y)\in E_2$, let
\begin{equation*}
  G_e = \bigcup_{i=0}^K B(Kx+i(y-x),K^\varepsilon )\subset \Z^d,
\end{equation*}
where we standardly identify $x=(x_1,x_2)\in \Z^2$ with
$(x_1,x_2,0,\dots,0)\in \Z^d$. $G_e$ is a thin parallelepiped of length
$K+2K^\varepsilon $ and width $2K^\varepsilon $ along the scaled edge
$Ke$.

Let $W^\star_e\subset W^\star$ (recall Section~\ref{s_preliminaries} for
  the notation) be the set of all doubly-infinite trajectories modulo
time shift that hit $G_e$ but not
$\bigcup_{e':\dist(e,e')\ge 1} G_{e'}$. The fact that the
co-dimension of $\Z^2$ is larger than $3$, that is, the random walk is
transient in the direction perpendicular to $\Z^2$, can be used to show
that
\begin{equation}
  \label{e:onlylocal}
  \mathbb P\bigg[ \parbox{8.5cm}{$\omega \in \Omega $ contains a trajectory with
      label smaller than $u$ that intersects $G_e$ but is not in
      $W^\star_e$}\bigg]\xrightarrow{K\to\infty}1.
\end{equation}

 From Lemma~2 of~\cite{RS11}, it follows that the probability that there
is a connection along the long direction of $G_e$ within $\I^u$ can be
made arbitrarily large by increasing~$K$. Let~$\mathcal C_e$ be the
probability that this connection uses only the trajectories in
$W^\star_e$. Using \eqref{e:onlylocal}, it follows  that the probability
of~$\mathcal C_e$ can be made arbitrarily large by increasing~$K$ too.

Moreover, Proposition~1 of \cite{RS11} (or our
  Proposition~\ref{p_connectedness}) can be used to show that for any
$x\in \Z^2\subset \Z^d$ and $v,w\in \I^u\cap B(Kx,K^\varepsilon )$ there
is connection of~$v$ and~$w$ within $\I^u\cap B(Kx,2K^\varepsilon )$ with
probability tending to~1 as~$K$ increases. We denote by~$\mathcal D_x$ the
event that this connection uses only the trajectories in
$\bigcup_{e\ni x} W^\star_e$. Again,
applying~\eqref{e:onlylocal} and choosing~$K$ large, the
probability of~$\mathcal D_x$ can be made arbitrarily large.

Finally, call the edge $e=(x,y)\in E^2$ good, when
$\mathcal C_e\cap \mathcal D_x\cap \mathcal D_y$ occur. It follows that
the probability that $e$ is good can be made arbitrarily large by
choosing $K$ large.  Since $W^\star_e$ and $W^\star_e$ are disjoint
subsets of $W^\star$ when $\dist(e,e')\ge 1$, the events `$e$ is good'
and `$e'$ is good' are independent when $\dist (e,e')\ge 4$.  Moreover,
using the events $\mathcal D_x$, the connections realising $\mathcal C_e$,
$\mathcal C_{e'}$ in two adjacent good edges $e$, $e'$ can be connected
to form one path.

Using the domination argument of~\cite{LSS97},
we see that for~$K$ large
the good edges dominate the supercritical Bernoulli percolation on~$\Z^2$,
in particular, there is with probability one an infinite cluster
$\mathcal C \subset \Z^2$ of good edges. Moreover, when
$x,y\in \mathcal C$ are connected by a path of length $\ell$ in
$\mathcal C$, $B(xK,K^\varepsilon )$ and $B(yK,K^\varepsilon )$ are
connected by a path of length at most
$\ell \{(K+2K^\varepsilon ) (2K^\varepsilon+1 )^{d-1}+2(4K^\varepsilon+1 )^d\}$
within $\I^u$ (the factor in braces is simply the volume of the
  parallelepiped $G_e$ plus volume of the two boxes $B(x,2K^\varepsilon)$,
  $B(y,2K^\varepsilon )$).

Theorem~1.1~of~\cite{AP96} then implies that the claim of
Theorem~\ref{t_ap} holds (with $\delta =1$) for all
$x \in \I^u\cap \bigcup_{y\in \mathcal C} B(y,N^\varepsilon )$. The
extension to all $x\in \I^u$ is then trivial by repeating the argument
for other coordinate directions and using Proposition~1 of~\cite{RS11}
or Proposition~\ref{p_connectedness} to made the final connections to
those $x$'s that are not in the $K^\varepsilon $-neighbourhood
of~$K\mathcal C$.

\bibliographystyle{alpha}
\bibliography{i_set}

\end{document}